\documentclass[a4paper,11pt]{amsart} 
\usepackage{amsmath}
\usepackage{amsfonts}
\usepackage{amssymb}
\usepackage{amsthm}
\usepackage{newlfont}
\usepackage{color}
\usepackage[pdftex]{hyperref}
\usepackage{subfig}
\usepackage{bbm}

\usepackage{mathrsfs}
\usepackage{palatino}  
\usepackage{fancyhdr}
\usepackage{graphicx}
\usepackage{geometry}
\usepackage{psfrag}

\newcommand{\remark}{\noindent{\it Remark. }}

\newcommand{\caA}{{\mathcal A}}
\newcommand{\caB}{{\mathcal B}}

\newcommand{\caF}{{\mathcal F}}
\newcommand{\caG}{{\mathcal G}}
\newcommand{\scrG}{{\mathcal G}}

\newcommand{\caL}{{\mathcal L}}
\newcommand{\caM}{{\mathcal M}}

\newcommand{\caO}{{\mathcal O}}

\newcommand{\caS}{{\mathcal S}}

\newcommand{\caU}{{\mathcal U}}
\newcommand{\caV}{{\mathcal V}}

\newcommand{\bbC}{{\mathbb C}}

\newcommand{\bbN}{{\mathbb N}}

\newcommand{\bbP}{{\mathbb P}}

\newcommand{\bbR}{{\mathbb R}}

\newcommand{\mfa}{\mathfrak{a}}
\newcommand{\mfb}{\mathfrak{b}}
\newcommand{\tmfa}{\tilde{\mfa}}
\newcommand{\tmfb}{\tilde{\mfb}}

\newcommand{\upm}{^{\caM}}
\newcommand{\sigz}{\Sigma_z}
\newcommand{\sigx}{\Sigma_x}

\newcommand{\ie}{{\it i.e.\/} }

\renewcommand{\Re}{\mathrm{Re}}
\renewcommand{\Im}{\mathrm{Im}}

\newcommand{\eps}{\epsilon}

\newcommand{\dd}{\mathrm{d}}

\newcommand{\Emean}[1]{\mathbb{E}(#1)}

\newcommand{\bra}{\left\langle}
\newcommand{\ket}{\right\rangle}

\newcommand{\beq}{\begin{equation}}
\newcommand{\eeq}{\end{equation}}
\newcommand{\beqn}{\begin{equation*}}
\newcommand{\eeqn}{\end{equation*}}
\newcommand{\baq}{ \begin{eqnarray} }
\newcommand{\eaq}{ \end{eqnarray} }
\newcommand{\mean}[1]{\mathbb{E}\left(#1\right)}
\newcommand{\prob}[1]{\mathbb{P}\left(#1\right)}

\newtheorem{thm}{Theorem}

\newtheorem{lma}[thm]{Lemma}

\begin{document}

\title[DMPK process]
{Existence of a unique strong solution to the DMPK equation}

\author[M. Butz]{Maximilian Butz}
\address{Fakult\"at f\"ur Mathematik \\ Technische Universit\"at M\"unchen \\
Boltzmannstr.~3 \\ 
85748 Garching, Germany}
\email{butz@ma.tum.de}

\date{\today }

\begin{abstract}

For the transmission of electrons in a weakly disordered strip of material Dorokhov, Mello, Pereyra and Kumar (DMPK) proposed a diffusion process for the transfer matrices. The correspoding transmission eigenvalues satisfy the DMPK stochastic differential equations, like Dyson Brownian motion in the context of GOE/GUE random matrices. We control the singular repulsion terms of this SDE with a stopping-time argument, and its degenerate initial condition by an approximation procedure, and thereby establish the DMPK equation to be well posed.


\end{abstract}

\maketitle


\section{introduction}

F. J. Dyson \cite{dyson} introduced an $N\times N$ matrix-valued Brownian motion, to be precise a stochastic process $\left(X_2(t)\right)_{t\geq0}$,

\begin{equation}
\label{dbmmat2}
 X_2(t)=Y_2(t)+Y_2(t)^*,
\end{equation}
with $Y_2(t)$ being an $N\times N$ matrix with all entries independent complex standard Brownian motions.

$X_2(t)$ takes values in the complex hermitian matrices, and is distributed like a scaled GUE matrix for fixed time $t$.
Using the It\^{o} formula and the invariance of the increments under unitary transformations, it is possible to prove that the law in path space of the eigenvalues $\lambda_1(t),...,\lambda_N(t)$ of $X_2(t)$ is the unique weak solution of the system of stochastic differential equations (SDE)
\begin{equation}
\label{dbmev}
 \dd \lambda_k(t)=\frac{{2}}{\sqrt{\beta}}\dd W_k(t)+2\sum_{l\neq k}\frac{1}{\lambda_k(t)-\lambda_l(t)}\dd t,\hspace{1cm}k=1,...,N
\end{equation}
with $W_k$ independent standard Brownian motions and $\beta=2$ in our case. Although it is not hard to formally derive (\ref{dbmev}), the singular level repulsion terms (which originate from second order perturbation theory) cause that a solution of (\ref{dbmev}) is only defined up to the first collision of two eigenvalues. Thus, to make the derivation of (\ref{dbmev}) from (\ref{dbmmat2}) rigorous, one needs to prove that almost surely no level crossing can occur. This was accomplished directly by analysis of (\ref{dbmmat2}) in \cite{mckean}, and in \cite{chan},\cite{rogersshi} via a Lyapunov function argument for (\ref{dbmev}) (for general $\beta\geq1$), which is presented comprehensively in \cite{agz}.

Among other applications in physics, random matrices have been employed to understand the properties of disordered conductors, in particular the phenomenon of universal conductance fluctuations \cite{as,imry}. In the arguably simplest case, one can model the conductor as a quasi-1D wire, a system of $N$ interacting channels. The matrix in consideration is the transfer matrix $\caM\in\bbC^{2N\times2N}$ which maps the quantum state at the left side of the wire to the state at the right side (rather than mapping incoming states to outgoing states, as a scattering matrix does).

Instead of the self-adjointness of Wigner matrices, the main constraint for transfer matrices is \emph{current conservation}, (corresponding to unitarity of the scattering matrix), which in a suitable basis reads
\beq
\label{conservesigz}
\caM^*\sigz\caM=\sigz,\hspace{1cm}\sigz=\begin{pmatrix}1_N&0\\0&-1_N\end{pmatrix},
\eeq
\cite{mpk}, and we define the group
\beq
\scrG_2=\left\lbrace\caM\in\bbC^{2N\times 2N}:\caM^*\sigz\caM=\sigz\right\rbrace.
\eeq
In case the underlying quantum mechanical system is \emph{time-reversal symmetric}, \ie{}in the absence of magnetic fields, the transfer matrix additionally has to satisfy
\beq
\label{conservesigx}
\sigx\caM\sigx=\overline{\caM},\hspace{1cm}\sigx=\begin{pmatrix}0&1_N\\1_N&0\end{pmatrix}.
\eeq
Accordingly,
\beq
\scrG_1=\left\lbrace\caM\in\bbC^{2N\times 2N}:\caM^*\sigz\caM=\sigz\mbox{ and }\sigx\caM\sigx=\overline{\caM}\right\rbrace.
\eeq
In addition to the groups $\scrG_\beta$ with parameter $\beta=1,2$ as defined above, the value $\beta=4$ is also algebraically and physically meaningful, as it corresponds to a quantum system allowing for spin-orbit scattering, \cite{beenakker}. However, this last case does not introduce any new stochastic features and is disregarded for the sake of notational simplicity. One should remark that the correspondence of symmetries of the quantum system and the choice of matrix ensembles already arises for Gaussian ensembles and Dyson Brownian motion \cite{mehta}, the ensemble suitable for time-reversal invariant quantum systems (and thus losely related to our $\caG_1$) being 
\begin{equation}
 \label{dbmmat1}
 X_1(t)=Y_1(t)+Y_1(t)^T,
\end{equation}
with $Y_1(t)$ being an $N\times N$ matrix with all entries independent \emph{real} standard Brownian motions.

As in the theory of Wigner matrices, the focus is on the spectral properties of those transfer matrices: for each $\caM$ there exists a decomposition
\begin{equation}
\label{singvalM}
 \caM=\begin{pmatrix} \caM_{++}&\caM_{+-}\\\caM_{-+}&\caM_{--}\end{pmatrix}=\begin{pmatrix} U_{+}&0\\0&U_{-}\end{pmatrix}\begin{pmatrix} \Lambda&\left(\Lambda^2-1\right)^{\frac{1}{2}}\\\left(\Lambda^2-1\right)^{\frac{1}{2}}&\Lambda\end{pmatrix}\begin{pmatrix} V_+&0\\0&V_-\end{pmatrix}
\end{equation}
with (non-unique) $N\times N$ unitary matrices $U_+$, $U_-$ or $V_+$, $V_-$, and a diagonal matrix $\Lambda\geq1_N$. For general $\caM\in\scrG_2$, there is no relationship between those unitaries, while for $\caM\in\scrG_1$, $U_+=\overline{U_-}$ and $V_+=\overline{V_-}$. For later use we define the submanifolds $\tilde{\caG}_\beta$ as the set of those matrices in $\caG_\beta$ with a decomposition such that $\Lambda>1_N$ and has nondegenerate eigenvalues.

For physically realistic conductors, $\caM$ will be a random quantity due to microscopic disorder, so that on a macroscopic level it is natural to model the transfer matrix as a stochastic process $\left(\caM(s)\right)_{s\geq0}$ depending on the length $s$ of the wire. 

If we combine two pieces of wire with transfer matrices $\caM_{\mathrm{I}}$ and $\caM_{\mathrm{II}}$, respectively, we expect the concatenated wire to have 
\begin{equation}
 \label{groupprop}
\caM=\caM_{\mathrm{II}}\caM_{\mathrm{I}}
\end{equation}
as its transfer matrix \cite{mpk} and if we furthermore assume that the transfer matrices of disjoint short pieces of wire are i.i.d., we arrive at the stochastic differential equation
\beq
\begin{split}
\label{evolMdiff}
\dd\caM(s)&=\dd\caL(s)\caM(s)\hspace{7mm}(s\geq0).
\end{split}
\eeq
with a suitable stochastic process $\left(\caL(s)\right)_{s\geq0}$ that encodes the scattering properties of short wire pieces and is chosen such that (\ref{conservesigz}) (and, for $\beta=1$, (\ref{conservesigx})) are conserved.
Stochastically, every intial distribution for (\ref{evolMdiff}) makes sense, but it is physically reasonable to assign to a zero-length conductor the transfer matrix
\beq
\begin{split}
\label{evolMstart}
 \caM(0)&={1}_{2N}.\\
\end{split}
\eeq
 The process $\caL$ we will actually consider is a $2N\times 2N$ matrix-valued Brownian motion
\begin{equation}
 \caL(s)=\left(\begin{array}{cc} \mathfrak{a}_+(s) & \mathfrak{b}(s) \\[1mm] \mathfrak{b}^*(s) &{ \mathfrak{a}_-}(s)\end{array}\right)\
\end{equation}
with $N\times N$ block processes $\mathfrak{a}_+$, $\mathfrak{a}_-$, $\mathfrak{b}$ which, in terms of (\ref{dbmmat2},\ref{dbmmat1}), are distributed like 
\begin{equation}
\label{abfromdbm}
\begin{tabular}{l l l l l}
 \ensuremath{\mfa_+}&\ensuremath{\stackrel{\dd}{=}}&\ensuremath{\mfa_-}&\ensuremath{\stackrel{\dd}{=}iX_2/\sqrt{2N}}&\\

\ensuremath{\Re\mfb}&\ensuremath{\stackrel{\dd}{=}}&\ensuremath{\Im\mfb}&\ensuremath{\stackrel{\dd}{=}X_1/(2\sqrt{N+1})}&\ensuremath{(\beta=1)}\\
&&\ensuremath{\mfb}&\ensuremath{\stackrel{\dd}{=}Y_2/\sqrt{N}}&\ensuremath{(\beta=2)}.
 \end{tabular}
\end{equation}
For $\beta=1$, $\Re\mfb$ and $\Im\mfb$ are mutually independent and jointly independent of $\mfa_\pm$, while $\mfa_+$ and $\mfa_-$ are correlated by $\mfa_-=\overline{\mfa_+}$. For $\beta=2$, however, the processes $\mfa_+$, $\mfa_-$ and $\mfb$ are all independent.

One motivation for this choice of $\caL$ is that it predicts the same dynamics for the diagonal matrix $\Lambda$ as the maximum entropy assumption \cite{mpk}. The reason is, that after endowing the groups $\scrG_\beta$ with a suitable right-invariant metric, (\ref{evolMdiff}) is just the Brownian motion on the Riemannian manifold $\scrG_\beta$, while \cite{mpk} essentially single out the heat kernel by their maximum entropy requirement for the propagator. Moreover, (\ref{evolMdiff}) has been derived in \cite{wdrsb,wdrsbmb} as scaling limit from a microscopic quantum model (an Anderson model on a tube) --- up to minor deviations originating from the Anderson Hamiltonian dictating a preferred basis in the weak-perturbation limit.

Finally, and most important in the present context, the increments of $\caL$ are invariant under the conjugation
\begin{equation}
\label{eq_dLinvar}
 \caU^*\dd\caL\caU\stackrel{d}{=}\dd\caL.
\end{equation}
$\caU$ can be any unitary matrix of the form
\begin{equation}
 \caU=\begin{pmatrix} U_{+}&0\\0&U_{-}\end{pmatrix}
\end{equation}
with the blocks $U_+$ and $U_-$ chosen independently for $\beta=2$, and correlated by $U_+=\overline{U_-}$ in case $\beta=1$, compare (21), (22) in \cite{wdrsb}. This perfectly corresponds to the orthogonal/unitary invariance of the GUE/GOE-like increments of (\ref{dbmmat2},\ref{dbmmat1}), and will finally allow us to prove an autonomous (\ie{}eigenvector-independent) equation for $\Lambda$.


First, we give a formal derivation and start with a stochastic evolution for the matrix $\caM_{++}^*\caM_{++}$ and its eigenvalues $\lambda_k$, ordered like $\lambda_1\geq...\geq\lambda_N$. Inserting the $++$ component of (\ref{evolMdiff}),
\begin{equation}
 \dd \caM_{++} =     \dd \mathfrak{a}  \caM_{++} + \dd \mathfrak{b}  \caM_{-+},
\end{equation}
into It\^{o}'s formula, we obtain
\begin{equation}
\begin{split}
 \dd (\caM^*_{++}\caM_{++})   =  &    \caM^*_{++} \left(  \dd \mathfrak{a}_+  \caM_{++} + \dd \mathfrak{b}  \caM_{-+}  \right) \\
 &+        \left(  \caM^*_{++}  \dd \mathfrak{a}_+^* + \caM_{-+}^* \dd \mathfrak{b}^*   \right)   \caM_{++}  \\
 & +   \left(  \caM^*_{++}  \dd \mathfrak{a}_+^* + \caM_{-+}^* \dd \mathfrak{b}^*   \right)  \left(  \dd \mathfrak{a}_+  \caM_{++} + \dd \mathfrak{b}  \caM_{-+}  \right)\\
=&\caM^*_{++} \dd\mfb  \caM_{-+}+\caM^*_{-+} \dd\mfb^*  \caM_{++}+\caM_{++}^*\caM_{++}\dd s+\caM_{-+}^*\caM_{-+}\dd s
\end{split}
\end{equation}
where we used the independence of $\mfa_+$ and $\mfb$, and their explicit form (\ref{abfromdbm}) in the last line. By second order perturbation theory for the eigenvalues of $\caM^*_{++}\caM_{++}$ and It\^o's formula,
\begin{equation}
 \dd\lambda_k=\bra v_k,\dd\left(\caM_{++}^*\caM_{++}\right)v_k\ket+\sum_{l\neq k}\frac{\left|\bra v_l,\dd\left(\caM_{++}^*\caM_{++}\right)v_k\ket\right|^2}{\lambda_k-\lambda_l},
\end{equation}
with $v_k$ the eigenvectors of $\caM^*_{++}\caM_{++}$. In view of (\ref{singvalM}), $v_k$ are the columns of $V_+^*$, while $\sqrt{\lambda_k}=\Lambda_{kk}$ are the singular values of $\caM_{++}$. Furthermore, by
\begin{equation}
\begin{split}
 U_-^*\caM_{-+}V_+^*=\left(\Lambda^2-1\right)^{\frac{1}{2}}\\
U_+^*\caM_{++}V_+^*=\Lambda
\end{split}
\end{equation}
it is convenient to consider the process with
\begin{equation}
 \label{deftildeb}
\dd\tilde{\mfb}(s)=U_+^*(s)\dd\mfb(s) U_-(s)
\end{equation}
for which, by (\ref{eq_dLinvar}),
\begin{equation}
\label{tilb=b}
 \dd\tilde{\mfb}\stackrel{d}{=}\dd\mfb.
\end{equation}
With these definitions,
\begin{equation}
\label{dlambda1}
\begin{split}
 \dd\lambda_k=&2\sqrt{\lambda_k(\lambda_k-1)}\dd\left(\Re\tilde{\mfb}_{kk}\right)+(2\lambda_k-1)\dd s
\\&+\sum_{l\neq k}\left(\sqrt{\lambda_l(\lambda_k-1)}\dd\tilde{\mfb}_{lk}+\sqrt{\lambda_k(\lambda_l-1)}\dd\overline{\tilde{\mfb}_{kl}}\right)\\
&\hspace{2cm}\times\left(\sqrt{\lambda_l(\lambda_k-1)}\dd\overline{\tilde{\mfb}_{lk}}+\sqrt{\lambda_k(\lambda_l-1)}\dd{\tilde{\mfb}_{kl}}\right)/{(\lambda_k-\lambda_l)}.
\end{split}
\end{equation}
We can use (\ref{tilb=b}) with (\ref{abfromdbm}) to see
\begin{equation}
\label{defBk}
 \Re\tilde{\mfb}_{kk}(s)=-\sqrt{\frac{1}{\beta(N-1)+2}}B_k(s)
\end{equation}
for all $k=1,...,N$, with independent standard real Brownian motions $B_k$, while
\begin{equation}
\begin{split}
 \left(\sqrt{\lambda_l(\lambda_k-1)}\dd\tilde{\mfb}_{lk}+\sqrt{\lambda_k(\lambda_l-1)}\dd\overline{\tilde{\mfb}_{kl}}\right)&\left(\sqrt{\lambda_l(\lambda_k-1)}\dd\overline{\tilde{\mfb}_{lk}}+\sqrt{\lambda_k(\lambda_l-1)}\dd{\tilde{\mfb}_{kl}}\right)\\
=\frac{\beta}{\beta(N-1)+2}&\left(2\lambda_k\lambda_l-\lambda_k-\lambda_l\right)\dd s,
\end{split}
\end{equation}
and thus
\begin{equation}
\label{SDElambda}
 \dd\lambda_k=\left(2\lambda_k-1+\frac{\beta}{\beta(N-1)+2}\sum_{l\neq k}\frac{2\lambda_k\lambda_l-\lambda_k-\lambda_l}{\lambda_k-\lambda_l}\right)\dd s-\sqrt{\frac{4\lambda_k(\lambda_k-1)}{\beta(N-1)+2}}\dd B_k.
\end{equation}
After introducing the \emph{transmission eigenvalues} $T_k=\lambda_k^{-1}\in[0,1]$, another application of the It\^o formula yields the DMPK (Dorokhov-Mello-Pereyra-Kumar) equation for the transmission eigenvalues (compare (3.9) in \cite{beenakker}, (24) in \cite{wdrsb}):
\begin{equation}
\label{dmpkeq}
 \dd T_k(s)=v_k(T(s))\dd s+D_k(T(s))\dd B_k(s),
\end{equation}
$B_k$, $k=1,...,N$ independent Brownian motions,
\begin{equation}
\label{driftdiff}
 \begin{split}
  v_k&=-T_k+\frac{2T_k}{\beta N+2-\beta}\left(1-T_k+\frac{\beta}{2}\sum_{j\neq k}\frac{T_k+T_j-2T_kT_j}{T_k-T_j}\right)\\
D_k&=\sqrt{4\frac{T_k^2(1-T_k)}{\beta N+2-\beta}}.
 \end{split}
\end{equation}
\remark{}As already noted in \cite{wdrsb}, the term ``DMPK equation'' usually refers to the forward equation for the density of the $T_k$'s. We also use it for the corresponding stochastic differential equation, as this is the more natural object in our analysis.

\noindent{}The form of both the drift term $v_k$ and the diffusion term $D_k$ causes the transmission values to decay to $0$ as $s$ tends to infinity, which perfectly matches the decrease of conductance with increasing wire length. However, the drift $v_k$ also contains repulsion terms originating from second order perturbation theory. As a consequence, the eigenvalues $T_k$ ``try to avoid'' degeneracy. Thus, and similar to Dyson Brownian motion, the above derivation is merely ``formal'', as the It\^o formula is only applicable if the denominator $\lambda_k-\lambda_l$ (or $T_k-T_l$) stays away from zero, \ie{}$\caM^*_{++}(s)\caM_{++}(s)$ never has degenerate eigenvalues. For $s>0$, this is already a nontrivial finding, but even more, we want to start the evolution equation (\ref{evolMdiff}) with the completely degenerate matrix (\ref{evolMstart}). The goal of our paper is the derivation of the following rigorous statement about the transmission value process of $\left(\caM(s)\right)_{s\geq0}$. 
\begin{thm}
\label{maintheorem}
 Let $\beta=1$ or $\beta=2$, and $\left(\caM(s)\right)_{s\geq0}$ be the solution of (\ref{evolMdiff}), starting from (\ref{evolMstart}). Then the distribution in pathspace of its transmission eigenvalue process is given by the law of the unique continuous process $\left(T_k(s)\right)_{s\geq0,k=1...N}$ which starts from $T_k(0)=1$ for $k=1,...,N$ and is a strong solution to (\ref{dmpkeq}) for $s>0$. Moreover, the transmission eigenvalues $T_k$, $T_l$ for $k\neq l$ almost surely never collide for $s>0$, with each staying in the open interval $(0,1)$.
\end{thm}
Speaking in terms of manifolds, we thus see that the paths of the Brownian motion (\ref{evolMdiff}) on $\caG_\beta$, even though starting from a degenerate matrix $\caM(0)\in\caG_\beta\setminus\tilde{\caG}_\beta$, are almost surely contained in $\tilde{\caG}_\beta$ for all positive $s$, and do not explode for finite $s$ (this would correspond to $T_k(s)=0$). It will be obvious from the proof of Theorem \ref{thmstartatone} that the way to handle the singular initial condition (\ref{evolMstart}) could as well be applied to any other degenerate $\caM(0)\in\caG_\beta\setminus\tilde{\caG}_\beta$. For the sake of notational simplicity, we stay with the physically most interesting case $\caM(0)=1_{2N}$.

We will start from considering the DMPK equation not as an eigenvalue process, but a process in its own right and show that it has a unique strong solution for all $\beta\geq1$ in subsection \ref{solveDMPK}. Once this is established, we verify in subsection \ref{equal_dist} that the transmission eigenvalue process of $\caM$ has the same law in path space.

Whenever possible, the analysis will be carried out in the picture of the transmission eigenvalues $T_k$ (rather than the $\lambda_k$). This is because the $T_k$ are the physically most intuitive quantity, each $T_k$ representing the transmittance of one channel. Accordingly, when modelling the passage of electrons through a $N$-channel wire, the dimensionless conductance of the wire is given by the Landauer formula (\cite{beenakker}, equation (33)),
\begin{equation}
 \label{landauer}
g=\sum_{k=1}^NT_k.
\end{equation}
This formula for $g$ and other linear functionals of the $T_k$ describe the most important transport properties of the disordered conductor. Plugging the DMPK forward equation into (\ref{landauer}), \cite{benrej} established the universal conductance fluctuations
\begin{equation}
 \mathrm{Var}(g)=\frac{2}{15\beta}
\end{equation}
for $\beta=2$ in the intermediate metallic phase with $1\ll s\ll N$.
\section{Rigorous analysis of the process of transmission eigenvalues}
\label{welldefdmpk}
\subsection{The unique strong solution of the DMPK equation}
\label{solveDMPK}

We first investigate equation (\ref{dmpkeq}) starting from a nondegenerate initial condition $0<T_1(0)<...<T_N(0)<1$, corresponding to $\caM(0)\in\tilde{\caG}_\beta$.
\begin{thm}
\label{thmnocoll}
Let $(\Omega,\bbP)$ be a probability space endowed with a filtration $\caF=\left(\caF_t\right)_{t\geq0}$ and $B_1,...,B_N$ be independent real Brownian motions adapted to $\caF$. Let $D_N$ be the set of strictly ordered $N$-tupels in $(0,1)$, $D_N=\left\lbrace x\in\bbR: 0<x_1<...<x_N<1\right\rbrace$. Assume initial conditions  $T(0)=\left(T_1(0),...,T_N(0)\right)\in D_N$, and take $\beta\geq1$. Then the stochastic differential equation
(\ref{dmpkeq}) has a unique strong solution with initial condition $T(0)$, and $T(s)\in D_N$ for all $s>0$. 
The law of this strong solution also defines the unique weak solution of (\ref{dmpkeq}). 
\end{thm}
\begin{proof}
 Following the ideas of \cite{agz}. For $R>1$, we regularize the singular repulsion in the drift term of (\ref{driftdiff}) by setting
\beq
\begin{split}
&v_k^{(R)}(T)=\chi_R(T)\left(-T_k+\frac{2T_k}{\beta N+2-\beta}\left(1-T_k+\frac{\beta}{2}\sum_{j\neq k}\frac{T_k+T_j-2T_kT_j}{T_k-T_j}\right)\right)\\
&D_k^{(R)}(T)=\chi_R(T)D_k(T)
\end{split}
\eeq
with $\chi_R$ a smooth cut-off function compactly supported in the interior of $D_N$ and $\chi_R(T)=1$ whenever $T\in D_N$ and $\mathrm{dist}(T,\partial D_N)\geq (\sqrt{2}R)^{-1}$.
Now $v_k^{(R)}$ and $D_k^{(R)}$ are both uniformly Lipschitz, so the modified system has a unique strong solution $\left(T^{(R)}(s)\right)_{s\geq0}$ 
and a unique weak solution 
(\cite{ks}, Theorems 5.2.5, 5.2.9, \cite{oksendal}, Lemma 5.3.1). The idea of the proof is that $T^{(R)}(s)$ should be a solution of the original DMPK equation (\ref{dmpkeq}) as long as the differences of the $T_k^{(R)}$ and their distance to the boundaries of $[0,1]$ are larger than $R^{-1}$. Taking $R$ to zero should then yield a solution to (\ref{dmpkeq}) for all times. To control the separation of the eigenvalues and their distance to $0$ and $1$, we define a Lyapunov function
\beq
\begin{split}
\label{definitionf}
f(T)=\sum_{k=1}^N\left(-2\log(|T_k|)-2\log(|1-T_k|)-\sum_{l=1,l\neq k}^N\log(|T_k-T_l|)\right)
\end{split}
\eeq

It will be enough to define $f$ for $T$ from $D_N$, as we will always use it together with stopping times which will make sure $T$ doesn't leave this set. 

Note that if we take $T\in D_N$, all the logarithms in (\ref{definitionf}) are positive, thus we have $f(T)\geq0$ and also 
\beq
-\log(|T_k-T_l|)\leq f(T).
\eeq
Define 
\beq
S_{M,R}=\inf\left\lbrace s\geq0:f\left(T^{(R)}(s)\right)\geq M\right\rbrace
\eeq
for $M>0$. As $f$ is smooth on sets where it is uniformly bounded, $S_{M,R}$ is a stopping time. On the event $\left\lbrace S_{M,R}>s\right\rbrace$, we have
\beq
-\log(|T_k^{(R)}(s)-T_l^{(R)}(s)|)\leq f(T^{(R)}(s))\leq M
\eeq
so if we take $M=M(R)=\log{R}$,
\beq
\begin{split}
\label{controlsep}
|T_k^{(R)}(s)-T_l^{(R)}(s)|\geq e^{-M}&=R^{-1}\\
|T_k^{(R)}(s)|\geq e^{-M}&=R^{-1}\\
|1-T_k^{(R)}(s)|\geq e^{-M}&=R^{-1}.
\end{split}
\eeq
Thus, up to $S_{M(R),R}$, $T^{(R)}$ and all $T^{(\tilde{R})}$ with $\tilde{R}\geq R$ coincide and $T^{(R)}$ also solves the original equation (\ref{dmpkeq}) up to $S_{M(R),R}$.

Therefore, to calculate $\dd f(s)$ up to the stopping time $S_M=S_{M(R),R}$ it is enough to plug (\ref{dmpkeq}) into It\^{o}'s formula (\cite{ks}, Theorem 3.3.3). We drop the overscript $(R)$, as it does not matter in this context.
\beq
\begin{split}
\dd f(T(s))&=\\
&2\sum_{k=1}^N\left(-\frac{1}{T_k}+\frac{1}{1-T_k}-\sum_{j=1,j\neq k}^N\frac{1}{T_k-T_j}\right)\\
&\hspace{1.8cm}\times\left(-T_k+\frac{2T_k}{\beta N+2-\beta}\left(1-T_k+\frac{\beta}{2}\sum_{l=1,l\neq k}^N\frac{T_k+T_l-2T_kT_l}{T_k-T_l}\right)\right)\dd s\\
&+\sum_{k=1}^N\left(\frac{1}{T_k^2}+\frac{1}{(1-T_k)^2}+\sum_{j=1,j\neq k}^N\frac{1}{(T_k-T_j)^2}\right)4\frac{T_k^2(1-T_k)}{\beta N+2-\beta}\dd s\\
&+\dd A(s),
\end{split}
\eeq
with $\left(A(s)\right)_{s\geq0}$ denoting the local martingale fulfilling
\begin{equation}
 \dd A(s)=2\sum_{k=1}^N\left(-\frac{1}{T_k}+\frac{1}{1-T_k}-\sum_{j=1,j\neq k}^N\frac{1}{T_k-T_j}\right)2T_k\sqrt{\frac{1-T_k}{\beta N +2-\beta}}\dd B_k(s).
\end{equation}
(Remark: Because of (\ref{controlsep}), $A(s\wedge S_M)$ is certainly a martingale, which is enough for our purpose. We only see that $A$ is really a local martingale after having verified that $S_M\nearrow\infty$ as $M\rightarrow\infty$.)

Introducing $T_0\equiv0$, $T_{N+1}\equiv1$, we can more conveniently rewrite $\dd f$
\beq
\begin{split}
\dd f(T(s))=&-\frac{4}{\beta N+2-\beta}\sum_{k=1}^N\left(\sum_{j=0,j\neq k}^{N+1}\frac{1}{T_k-T_j}\right)\\&\hspace{2cm}\times\left((\beta-1)T_k^2+\beta N T_k(T_k-1)+\beta\sum_{l=0,l\neq k}^{N+1}\frac{\rho_k}{T_k-T_l}\right)\dd s\\
&+\sum_{k=1}^N\left(\sum_{j=0,j\neq k}^{N+1}\frac{1}{(T_k-T_j)^2}\right)4\frac{\rho_k}{\beta N+2-\beta}\dd s\\
&+\dd A(s),
\end{split}
\eeq
where we have set $\rho_k=T_k^2(1-T_k)$. Using $\rho_0=\rho_{N+1}=0$, we have
\begin{equation}
\begin{split}
 Z=&\sum_{k=1}^{N}\left(\left(\sum_{j=0,j\neq k}^{N+1}\frac{1}{T_k-T_j}\right)\left(\sum_{l=0,l\neq k}^{N+1}\frac{\rho_k}{T_k-T_l}\right)-\sum_{m=0,m\neq k}^{N+1}\frac{\rho_k}{(T_k-T_m)^2}\right)\\
=&\sum_{k=0}^{N+1}\left(\left(\sum_{j=0,j\neq k}^{N+1}\frac{1}{T_k-T_j}\right)\left(\sum_{l=0,l\neq k}^{N+1}\frac{\rho_k}{T_k-T_l}\right)-\sum_{m=0,m\neq k}^{N+1}\frac{\rho_k}{(T_k-T_m)^2}\right)\\
=&\sum_{k=0}^{N+1}\sum_{j=0}^{N+1}\sum_{l=0}^{N+1}\mathbf{1}\left(l\neq k, j\neq k, l\neq j\right)\frac{\rho_k}{(T_k-T_l)(T_k-T_j)}\\
=&\sum_{k=0}^{N+1}\sum_{j=0}^{N+1}\sum_{l=0}^{N+1}\mathbf{1}\left(l\neq k, j\neq k, l\neq j\right){\rho_k}\left(\frac{1}{T_k-T_j}-\frac{1}{T_k-T_l}\right)\frac{1}{T_j-T_l}\\
=&-2Z+\sum_{k=0}^{N+1}\sum_{j=0}^{N+1}\sum_{l=0}^{N+1}\mathbf{1}\left(l\neq k, j\neq k, l\neq j\right)\left(\frac{\rho_k-\rho_j}{T_k-T_j}-\frac{\rho_k-\rho_l}{T_k-T_l}\right)\frac{1}{T_j-T_l}\\
=&\frac{1}{3}\sum_{k=0}^{N+1}\sum_{j=0}^{N+1}\sum_{l=0}^{N+1}\mathbf{1}\left(l\neq k, j\neq k, l\neq j\right)(1-T_j-T_l-T_k),
\end{split}
\end{equation}
and thus for $T\in D_N$
\begin{equation}
 |Z|\leq\frac{2}{3}(N+1)^3.
\end{equation}
Furthermore
\beq
\sum_{k=1}^{N}\sum_{j=0,j\neq k}^{N+1}\frac{T_k(1-T_k)}{T_k-T_j}=\sum_{k=0}^{N+1}\sum_{k=0}^{N+1}\mathbf{1}(j\neq k)\frac{T_k(1-T_k)-T_j(1-T_j)}{2(T_k-T_j)}=\frac{1}{2}\sum_{k\neq j}(1-T_k-T_j),
\eeq
so
\beq
\left|\sum_{k=1}^{N}\sum_{j=0,j\neq k}^{N+1}\frac{T_k(1-T_k)}{T_k-T_j}\right|\leq\frac{(N+1)(N+2)}{2}.
\eeq
for all $T\in D_N$. Finally,
\beq
\begin{split}
\sum_{k=1}^{N}\sum_{j=0,j\neq k}^{N+1}\frac{T_k^2}{T_j-T_k}&=-\sum_{k=1}^{N}T_k+\sum_{k=1}^{N}\frac{T_k^2}{1-T_k}+\sum_{j,k=1,j\neq k}^N\frac{T_k^2}{T_j-T_k}\\&=-\sum_{k=1}^N\sum_{j=1}^N\frac{T_k+T_j}{2}+\sum_{k=1}^N\frac{T_k^2}{1-T_k}
\end{split}
\eeq
and thus on $D_N$
\beq
-N^2\leq\sum_{k=1}^{N}\sum_{j=0,j\neq k}^{N+1}\frac{T_k^2}{T_j-T_k}\leq\sum_{k=1}^N\frac{T_k^2}{1-T_k}.
\eeq

Thus there is a continous bounded function $g$ with $g(T)\leq C N^2$ for $T\in D_N$ such that
\beq
\begin{split}
\dd &f(T(s))\\&=g(T(s))\dd s+\frac{4(\beta-1)}{\beta N +2 -\beta}\sum_{k=1}^N\left(\frac{T_k^2}{1-T_k}-\left(\sum_{j=0,j\neq k}^{N+1}\frac{T_k^2(1-T_k)}{(T_k-T_j)^2}\right)\right)\dd s+\dd A(s)\\
&=g(T(s))\dd s-\frac{4(\beta-1)}{\beta N +2 -\beta}\sum_{k=1}^N\left(\sum_{j=0,j\neq k}^{N}\frac{T_k^2(1-T_k)}{(T_k-T_j)^2}\right)\dd s+\dd A(s).
\end{split}
\eeq

Recalling that $\beta\geq1$ and $A\left(s\wedge S_M\right)$ is a martingale, 
\beq
\Emean{f(T(s\wedge S_M))}\leq CN^2 s+f(T(0)).
\eeq
As $f$ is nonnegative, we have the Markov estimate
\beq
\prob{S_M\leq s}\leq\frac{1}{M}\Emean{f(T(s\wedge S_M))}\leq\frac{CN^2 s+f(T(0))}{M}
\eeq
for all $s\geq0$ and $M>0$. 

Now as by definition
\beq
S_{\tilde{M}}=S_{M(\tilde{R}),\tilde{R}}\geq S_{M({R}),{R}}=S_M
\eeq
for $\tilde{R}\geq R$, we conclude from $\prob{S_M\leq s}\rightarrow0$ as $M\rightarrow\infty$ for all $s\geq0$ that
\begin{equation}
 S_M\nearrow\infty\hspace{6mm}(M\rightarrow\infty)
\end{equation}
almost surely. As we have already seen that $T^{(R)}$ is a solution to (\ref{dmpkeq}) up to $S_{M(R),R}$, setting $T=T^{(R)}$ up to $S_{M(R),R}$ for all $R>1$ provides us with a strong solution of (\ref{dmpkeq}) for $s\in[0,\infty)$. This finishes the strong existence part.

To prove strong uniqueness, let a strong solution $T$ of (\ref{dmpkeq}) be given and define the stopping time $\tilde{S}_M$ quite analogously to $S_M$, but this time based on $T$ instead of $T^{(R)}$,
\beq
\tilde{S}_{M}=\inf\left\lbrace s\geq0:f\left(T(s)\right)\geq M\right\rbrace.
\eeq
If we take $e^M=R$, we have a result like (\ref{controlsep}), and by strong uniqueness of $T^{(R)}$, $T$ and $T^{(R)}$ coincide up to $\tilde{S}_M$. As one can show
\beq
\tilde{S}_M\nearrow\infty\hspace{6mm}(M\rightarrow\infty)
\eeq
just as before, $T$ has to be the solution we obtained in the existence part almost surely for all times.

Weak uniqueness of $T$ follows by exactly the same reasoning from the weak uniqueness property of $T^{(R)}$ for all $R>1$.

Finally, $T(s)\in D_N$ for all $s\geq0$ almost surely, because we have seen that for any $s\geq0$ there is almost surely a (random) $R>1$ such that $T(t)=T^{(R)}(t)$ for all $t\leq s$, but the $T^{(R)}$ stay in $D_N$ almost surely due to the cutoff in the definitions of $v_k^{(R)}$ and $D_k^{(R)}$.
\end{proof}
Next, we want to extend our existence and uniqueness result to the initial condition that corresponds to $\caM(0)={1}_{2N}$, i.e. $T_k(0)=1$ for all $k=1,...,N$. To do so, we adapt Proposition 4.3.5 and Lemma 4.3.6 about Dyson Brownian Motion in \cite{agz} to the DMPK equation (\ref{dmpkeq}).

To control the convergence of approximate solutions, we start with a comparison result for solutions with different initial conditions in $D_N$.

\begin{lma}
\label{lmaconserveorder}
 Fix $\beta\geq1$, $N\in\bbN$ and let $\left(R(s)\right)_{s\geq0}$ and $\left(T(s)\right)_{s\geq0}$ be two strong solutions to (\ref{dmpkeq}) with $R(0),T(0)\in D_N$ and 
\begin{equation}
\label{initialordering}
 R_{j}(0)>T_j(0)\hspace{5mm}\mbox{ for all }j=1,...,N.
\end{equation}
Then we almost surely have for all $s>0$ and all $j=1,...,N$
\begin{equation}
\label{ordering}
 R_{j}(s)>T_j(s)\hspace{5mm}\mbox{ for all }j=1,...,N.
\end{equation}
\end{lma}
\begin{proof}
 We start with defining the stopping time
\beq
S_\eps=\inf\left\lbrace s\geq0:\min_{j}\min\lbrace1-R_j(s),1-T_j(s)\rbrace<\eps\right\rbrace
\eeq
for $\eps>0$.
As we have seen in the proof of Theorem \ref{thmnocoll}, 
\beq
S_\eps\rightarrow\infty\hspace{5mm}\mbox{a.s.}
\eeq
as $\eps\rightarrow0$.

Furthermore let the stopping time $S$ be when the ordering (\ref{ordering}) is violated first, i.e.
\beq
S=\inf\left\lbrace s\geq0:R_j(s)=T_j(s)\mbox{ for some }j\right\rbrace.
\eeq
We want to show that $S=\infty$ a.s.

From (\ref{driftdiff}) we see that $\sum_kv_k(T)$ is a polynomial in the $T_k$'s so that
\begin{equation}
 \left|\sum_kv_k(R(s))-\sum_kv_k(T(s))\right|\leq C \sum_{k=1}^N\left|R_k(s)-T_k(s)\right|=C \sum_{k=1}^N\left(R_k(s)-T_k(s)\right)
\end{equation}
on the event $\lbrace S\geq s\rbrace$, where we have used that all $T_j,R_j$ are contained in the bounded set $[0,1]$.
Here and in subsequent estimates we use constants $C, \tilde{C}...$ which might depend on $N$, but not on $\eps$.
On the other hand, for
\beq
D_k(T_k)=2\sqrt{\frac{T_k^2(1-T_k)}{\beta N+2-\beta}}
\eeq
we have $\left|D_k(R_k)-D_k(T_k)\right|\leq C\left|R_k-T_k\right|/\sqrt{\eps}$ as long as $T_k,R_k\leq1-\eps$, so
\beq
\left|\sqrt{\sum_k\left(D_k\left(R_k(s)\right)-D_k\left(T_k(s)\right)\right)^2}\right|\leq\frac{\tilde{C}}{\sqrt\eps}\sum_{k=1}^N\left(R_k(s)-T_k(s)\right)
\eeq
on $\lbrace S\wedge S_\eps\geq s\rbrace$.

Thus 
\beq
u(s)=\sum_{k=1}^N\left(R_k(s)-T_k(s)\right)
\eeq
solves the stochastic differential equation
\beq
\dd u(s)=f(R(s),T(s))\dd s+\sigma(R(s),T(s))\dd B(s)
\eeq
with $B$ a Brownian motion and 
\begin{equation}
\label{estfandsig}
 \begin{split}
  \left|f(R(s),T(s))\right|&\leq C u\\
 \left|\sigma(R(s),T(s))\right|&\leq \frac{\tilde{C}}{\sqrt\eps} u
 \end{split}
\end{equation}
on $\lbrace S\wedge S_\eps\geq s\rbrace$. By It\^{o}'s rule, 
\beq
\dd\log(u(s))=\frac{1}{u(s)}\dd u(s)-\frac{1}{2u(s)^2}\sigma^2\dd s=h(R,T)\dd s+\tilde{\sigma}(R,T)\dd B(s)
\eeq
with $h\geq -C-\tilde{C}^2/\eps$ and $|\tilde{\sigma}|\leq \tilde{C}/\sqrt{\eps}$. Thus, on the event $\lbrace S\wedge S_\eps\geq s\rbrace$,
\beq
\log(u(s))\geq\log(u(0))-(C+\tilde{C}^2/\eps)s+\int_0^s\tilde{\sigma}(R(t),T(t))\dd B(t).
\eeq
As 
\begin{equation}
 \int_0^{S\wedge S_\eps\wedge s}\tilde{\sigma}^2\dd t\leq \frac{\tilde{C}^2s}{\eps},
\end{equation}
we have
\beq
\int_0^s\tilde{\sigma}(R,T)\dd B>-\infty
\eeq
almost surely on $\lbrace S\wedge S_\eps\geq s\rbrace$, and thus from (\ref{initialordering})
\begin{equation}
 u(s)\geq u(0)\exp\left(-(C+\tilde{C}^2/\eps)s+\int_0^s\tilde{\sigma}(R,T)\dd B\right)>0
\end{equation}
almost surely on $\lbrace S\wedge S_\eps\geq s\rbrace$ for all $\eps>0$. So with $S_\eps\rightarrow\infty$ as $\eps\rightarrow0$, $u(s)>0$ almost surely whenever $s\leq S$.

\vspace{5mm}

Now assume that for some finite $s_0$, $\prob{S\leq s_0}>0$. And define $u_j(s)=R_j(s)-T_j(s)$, which is almost surely nonnegative for all $s\leq S$ by continuity of paths. Whenever $S<\infty$, there is a (random) $k\in\lbrace1,...,N\rbrace$ such that $u_k(S)=0$, but as we have seen above that $u(s)>0$ almost surely on $\lbrace s\leq S\rbrace$, there is almost surely a $j$ such that $R_j(S)-T_j(S)>0$. 

For all $k$, $u_k$ solves the stochastic differential equation
\beq
\begin{split}
\dd& u_k(s)\\&=\left(-\frac{\beta(N-1)}{\beta(N-1)+2}u_k(s)-\frac{2}{\beta(N-1)+2}(R_k(s)+T_k(s))u_k(s)+r(R(s),T(s))\right)\dd s
\\&\hspace{2cm}+2\left(\sqrt{\frac{R_k^2(1-R_k)}{\beta N+2-\beta}}-\sqrt{\frac{T_k^2(1-T_k)}{\beta N+2-\beta}}\right)\dd B_k(s),
\end{split}
\eeq
with $r$ denoting the difference of the repulsion terms
\beq
\begin{split}
r(R,T)=&\frac{\beta}{\beta N+2-\beta}\sum_{j\neq k}\left[R_k\left(\frac{R_k+R_j-2R_kR_j}{R_k-R_j}\right)-T_k\left(\frac{T_k+T_j-2T_kT_j}{T_k-T_j}\right)\right]\\
=&\frac{\beta}{\beta N+2-\beta}\sum_{j\neq k}\frac{R_kT_k(R_k-T_k)+2T_jR_j(R^2_k-T^2_k)+(R_k+T_k)T_k(R_j-T_j)}{(R_k-R_j)(T_k-T_j)}\\
&\hspace{5mm}+\frac{\beta}{\beta N+2-\beta}\sum_{j\neq k}\frac{T_j\left(T_k+R_k+R_j+2R_kT_k\right)(T_k-R_k)}{(R_k-R_j)(T_k-T_j)},
\end{split}
\eeq
$s$ dependence suppressed. In the last two lines, the denominator $(R_k-R_j)(T_k-T_j)$ is always positive as $R,T\in D_N$. Now on a path with $S<\infty$, $s\leq S$, and $k$ such that $u_k(S)=0$, the numerator in the second last line is positive, the one in the last line is negative. However, while the negative factor $(T_k-R_k)(s)=-u_k(s)$ vanishes as $s\nearrow S$, some of the positive terms $R_j-T_j$ are bounded away from zero uniformly in a neighborhood of $S$, so there is an ($\omega$-dependent) $s_1<S$ such that
\beq
r(R(s),T(s))\geq0\hspace{5mm}\mbox{ for all }s\in\left[s_1,S\right].
\eeq
Then given our assumption that $\prob{S(\omega)\leq s_0}>0$, we can find \emph{deterministic}, constant $s_1$, $\eps$ and $k$, such that the event 
\begin{equation}
 A=\lbrace s_1< S\leq s_0\rbrace\cap\lbrace r(R(s),T(s))\geq 0 \forall s\in[s_1,S]\rbrace\cap\lbrace u_k(S)=0\rbrace\cap\lbrace S_\eps\geq s_0\rbrace
\end{equation}
has a positive probability $\prob{A}>0$. On the event $A$ we have for $s\in[s_1,S]$
\beq
\dd u_k(s)\geq -2u_k(s)\dd s+\sigma_k\left(R_k(s),T_k(s)\right)\dd B_k(s)
\eeq
with
\beq
\left|\sigma_k(R_k(s),T_k(s))\right|\leq \frac{{C}}{\sqrt\eps} u_k(s).
\eeq
As before we obtain
\begin{equation}
 u_k(s)\geq u_k(s_1)\exp\left(-(2+C^2/\eps)(s-s_1)+\int_{s_1}^s\tilde{\sigma}_k\dd B_k\right)
\end{equation}
for all $s\in\left[s_1,S\right]$. Again, we have $\tilde{\sigma}_k^2\leq C^2/\eps$, and thus
\begin{equation}
\int_{s_1}^{S}\tilde{\sigma}_k\dd B_k>-\infty
\end{equation}
for almost surely on $A$. Together with $u_k(s_1)>0$ on $A$ this means for $u_k(S)>0$ on $A$, contradicting the assumption $u_k(S)=0$. So our initial assumption is wrong and $S(\omega)$ has to be infinite for almost all $\omega\in\Omega$. This means that the initial ordering (\ref{initialordering}) stays conserved almost surely.
\end{proof}

\begin{thm}
\label{thmstartatone}
 Let $\beta\geq1$, $N\in\bbN$, and $T_k(0)=1$, $k=1,...,N$. Then there is a unique continuous $\left(T(s)\right)_{s\geq0}$ such that $T(s)\in D_N$ for all $s>0$ and $\left(T(s)\right)_{s>0}$ is a strong solution to the DMPK equation (\ref{dmpkeq}). Again, the law of this strong solution also defines the unique weak solution.
\end{thm}
\begin{proof}
 To approximate $T(0)$ by intial conditions $T^{(n)}(0)\in D_N$, define
\begin{equation}
 T^{(n)}_k(0)=1-\frac{N+1-k}{n}
\end{equation}
for all $n\geq N+1$. By Theorem \ref{thmnocoll}, those initial conditions allow for a solution $\left(T^{(n)}(s)\right)_{s\geq0}$ of (\ref{dmpkeq}) with $T^{(n)}(s)\in D_N$ for all $n$, $s$. By Lemma \ref{lmaconserveorder},
\begin{equation}
 T^{(m)}_k(s)<T^{(n)}_k(s)
\end{equation}
almost surely for all $N+1\leq m<n$, $s\geq0$, so the $T^{(n)}_k$ converge from below to a limit process $\left(T(s)\right)_{s>0}$ with $T_k(0)=1$ for all $k$ and $T_k(s)\in \overline{D_N}$ for all $s\geq0$.

This convergence result can be considerably strengthened. Define for $m<n$
\beq
u^{(m,n)}(t)=\sum_{k=1}^N\left|{T}^{(n)}_k(t)-T^{(m)}_k(t)\right|=\sum_{k=1}^N\left({T}^{(n)}_k(t)-T^{(m)}_k(t)\right)\geq0,
\eeq
which solves the stochastic differential equation
\beq
\dd u^{(m,n)}(t)=f\left(T^{(m)}(t),T^{(n)}(t)\right)\dd t+\sigma\left(T^{(m)}(t),T^{(n)}(t)\right)\dd B^{(m,n)}(t)
\eeq
with $B^{(m,n)}$ a Brownian motion. This time we estimate 
\begin{equation}
\label{fsigest}
 \begin{split}
  \left|f\left(T^{(m)}(t),T^{(n)}(t)\right)\right|&\leq C u^{(m,n)}(t)\\
 \left|\sigma\left({T}^{(m)}(t),T^{(n)}(t)\right)\right|&\leq{\tilde{C}}\sqrt{u^{(m,n)}(t)}
 \end{split}
\end{equation}
for all $t$, and only $N$-dependent constants $C,\tilde{C}$. Then the stochastic process
\beq
\begin{split}
u^{(m,n)}(s)-u^{(m,n)}(0)-\int_0^sf\left(T^{(m)}(t),T^{(n)}(t)\right)\dd t&=\int_0^s\sigma\left(T^{(m)}(t),T^{(n)}(t)\right)\dd B^{(m,n)}(t)\\
&=X_{m,n}(s)
\end{split}
\eeq
is a continuous martingale and its square $Y_{m,n}=X_{m,n}^2$ is a continuous submartingale.

As $u\leq N$, $\sigma^{m,n}(t)=\sigma\left(T^{(m)}(t),T^{(n)}(t)\right)$ is bounded, and as $T^{(m)}_k(s)\rightarrow T_k(s)$ from below for all $k$, $s$ as $m\rightarrow\infty$, we also have
\beq
\sup_{n> m}\left|\sigma^{m,n}(t)\right|\rightarrow 0 \hspace{5mm}(m\rightarrow\infty)
\eeq
almost surely and for all $t\geq0$. Using this and dominated convergence in the It\^{o} isometry, we obtain
\beq
\kappa_m=\sup_{n> m}\mean{Y_{m,n}(s)}=\sup_{n> m}\int_0^{s}\mean{\left|\sigma^{m,n}(t)\right|^2}\dd t\rightarrow0
\eeq
as $m\rightarrow\infty$.

The key idea for uniform convergence is Theorem 1.3.8 in \cite{ks}, which states that for any $\alpha>0$, $s\geq 0$
\beq
\sup_{n> m}\prob{\sup_{0\leq t\leq s} Y_{m,n}(t)\geq\alpha}\leq \sup_{n> m}\frac{\mean{Y_{m,n}(s)}}{\alpha}=\frac{\kappa_m}{\alpha}.
\eeq
Choosing $\eps_j>0$ and a subsequence $m_j\rightarrow\infty$ such that 
\beq
\sum_{j=1}^\infty\left(\frac{\kappa_{m_j}}{\eps_j}+\sqrt{\eps_j}\right)<\infty
\eeq
we obtain
\beq
\sum_{j=1}^{\infty}\prob{\sup_{0\leq t\leq s} Y_{m_j,m_{j+1}}(t)\geq\eps_j}\leq\sum_{j=1}^{\infty}\frac{\kappa_{m_j}}{\eps_j}<\infty.
\eeq
By the Borel-Cantelli lemma we thus have an almost surely finite $J(\omega)$ such that for all $j\geq J$
\beq
\sup_{0\leq t\leq s} Y_{m_j,m_{j+1}}(t)<\eps_j
\eeq
and thus for all $t\leq s$, $j\geq J$
\begin{equation}
\begin{split}
 u^{({m_j},{m_{j+1}})}(t)&\leq\sqrt{\eps_j}+u^{({m_j},{m_{j+1}})}(0)+C\int_0^tu^{({m_j},{m_{j+1}})}(\tau)\dd\tau\\
u^{({m_j},{m_{j+1}})}(t)&\leq\left(\sqrt{\eps_j}+u^{({m_j},{m_{j+1}})}(0)\right)e^{Ct},
\end{split}
\end{equation}
where we used the Gronwall inequality. With the monotonicity of the $T^{(m)}_k$ in $m$ and thus of $u^{(m,n)}$, we almost surely have
\beq
\begin{split}
\lim_{m\rightarrow\infty}\sup_{n\geq m}\sup_{0\leq t \leq s}u^{(m,n)}(t)&=\lim_{l\rightarrow\infty}\sup_{0\leq t \leq s}\sum_{j\geq l}u^{({m_j},{m_{j+1}})}(t)\\
&\leq\lim_{l\rightarrow\infty}e^{Cs}\left(\sum_{j\geq l}\sqrt{\eps_j}+\sum_{k=1}^N\left({T}_k(0)-T^{(m_l)}_k(0)\right)\right)=0
\end{split}
\eeq
and thus almost surely uniform convergence of $T^{(m)}_k$ to $T_k$ on compact intervals for all $k$. (This means the speed of convergence is uniform in compact intervals $[0,s]$, but might still depend on the particular path). In particular, all $T_k$ are continuous almost surely.

Now assume for a moment that we already know that for a certain $s_0>0$, $T(s_0)\in D_N$ almost surely. Then denote by $\left(\tilde{T}(t)\right)_{t\geq s_0}$ the unique strong solution of (\ref{dmpkeq}) starting at $s_0$ with $\tilde{T}(s_0)=T(s_0)$, which exists by Theorem \ref{thmnocoll}. 

\noindent{}Define $\tilde{S}_\eps=\inf\left\lbrace s\geq s_0:\tilde{T}_N(s)\geq1-\eps\right\rbrace$ for $\eps>0$ and 
\beq
u^{(m)}(t)=\sum_{k=1}^N\left(\tilde{T}_k(t)-T^{(m)}_k(t)\right)
\eeq
we have functions $f$, $\sigma$ and a Brownian motion $B^{(m)}$ such that
\begin{equation}
 \dd u^{(m)}(s)=f\left(\tilde{T}(s),T^{(m)}(s)\right)\dd s+\sigma\left(\tilde{T}(s),T^{(m)}(s)\right)\dd B^{(m)}(s)
\end{equation}
with the same estimates as in (\ref{estfandsig})
\begin{equation}
\begin{split}
 \left|f\left(\tilde{T}(s),T^{(m)}(s)\right)\right|&\leq C u^{(m)}(s)\\
\sigma\left(\tilde{T}(s),T^{(m)}(s)\right)&\leq \tilde{C}\frac{u^{(m)}(s)}{\sqrt{\eps}}
\end{split}
\end{equation}
on the event $\left\lbrace s_0\leq s\leq \tilde{S}_\eps\right\rbrace$. From these estimates and $\mean{u^{(m)}(s_0)^2}\rightarrow 0$ as $m\rightarrow\infty$, one can conclude by the standard Gronwall method that
\begin{equation}
 \mean{u^{(m)}\left(s\wedge \tilde{S}_\eps\right)^2}\rightarrow 0
\end{equation}
as $m\rightarrow0$ for all $s>s_0$ and all $\eps>0$. By the monotonicity of $u^{(m)}$ in $m$ we therefore also have $u^{(m)}\left(s\wedge \tilde{S}_\eps\right)\rightarrow0$ almost surely for all $s>s_0$ and $\eps>0$, and as $\tilde{S}_\eps\nearrow0$ almost surely as $\eps\rightarrow0$, even $u^{(m)}\left(s\right)\rightarrow0$ almost surely for all $s$. Thus $\left(T(t)\right)_{t\geq s_0}$ and $\left(\tilde{T}(t)\right)_{t\geq s_0}$ are modifications of each other, and by their continuity, they are indistinguishable.

To verify that $\left(T(s)\right)_{s>0}$ is a solution to (\ref{dmpkeq}), the only thing left to prove is that there is almost surely no $s_1>0$ such that $T(s)\in \overline{D_N}\setminus D_N$ for $s\in\left[0,s_1\right)$. First, almost surely $T_k(s)>T^{(1)}_k>0$ for all $s>0$, $k=1,...,N$, by Theorem \ref{thmnocoll} and Lemma \ref{lmaconserveorder}. Next, assume $0<T_k(s)<1$ for all $s>0$, and all $k$, but that with positive probability there is a neighborhood of zero where some of the components of $T$ ``touch''. Then, due to the continuity of $T$ as a locally uniform limit of continuous functions, it is possible to find a deterministic $k$, a fixed open interval $(a,b)\subset\bbR^+$ and a  fixed $\delta>0$ such that on an event $\Omega'$ with $\prob{\Omega'}>0$ for all $s\in(a,b)$, $T_k(s)=T_{k+1}(s)$, but $T_{k+1}(s)<T_{k+l}(s)-\delta$ for all $l\geq 2$, and $T_N(s)<1-\delta$. On this event $\Omega'$, with the uniform convergence of $T^{(n)}$ to $T$, we have for sufficiently large $n$, for all $s\in(a,b)$
\begin{equation}
 \dd T_{k+1}^{(n)}=f\left(T^{(n)}\right)\dd s+\sigma\left(T^{(n)}\right)\dd B_{k+1}(s)
\end{equation}
with $\sigma$ uniformly bounded but $f$ becoming more and more singular because of the $T_{k+1}^{(n)}$ and $T_{k}^{(n)}$ repulsion
\begin{equation}
 f\left(T^{(n)}\right)\geq C\left(\eps_n^{-1}-\delta^{-1}-1\right)
\end{equation}
with $\delta$ fixed while $\eps_n\rightarrow0$. Thus, with probability $\prob{\Omega'}>0$, there is an $n\in\bbN$ such that 
\begin{equation}
 T_{k+1}^{(n)}(b)>1
\end{equation}
which is not possible for a solution of (\ref{dmpkeq}) with intial conditions in $D_N$ due to Theorem \ref{thmnocoll}.

Finally, we have to exclude that there is a neighborhood of $0$ such that $T_N(s)=1$ for all $s\in\left[0,s_1\right)$. If so, there was an open interval $(a,b)$ again, a $k$ and a $\delta>0$ such that on a positive probability set $\Omega''$, $T_k(s)=1$ for all $s\in(a,b)$, but $T_{j}<1-\delta$ for all $j<k$. Then almost surely in $\Omega''$ we have the following convergences in
\begin{equation}
\begin{split}
 v_k\left(T^{(n)}\right)(s)&\leq-T^{(n)}_k+\frac{2T^{(n)}_k}{\beta N+2-\beta}\left(1-T^{(n)}_k+\frac{\beta}{2}\sum_{j< k}\frac{T^{(n)}_k+T^{(n)}_j-2T^{(n)}_kT^{(n)}_j}{T^{(n)}_k-T^{(n)}_j}\right)\\
&\rightarrow-1+\frac{\beta(k-1)}{\beta N+2-\beta}\\
&\leq-\frac{2}{\beta N+2-\beta}
\end{split}
\end{equation}
and 
\begin{equation}
 D_k\left(T^{(n)}\right)(s)\rightarrow 0
\end{equation}
for all $s\in(a,b)$ as $n\rightarrow\infty$. By It\^{o}'s isometry and dominated convergence,
\beq
\int_a^bD_k\left(T^{(n)}\right)(s)\dd B_k(s)\rightarrow 0
\eeq
in $L^2(\Omega'',\bbP)$ as $n\rightarrow0$ and thus almost surely on a subsequence $n_j$, but then
\begin{equation}
T_k(s)= \lim_{j\rightarrow\infty}T_k^{(n_j)}(s)\leq1-\frac{2}{\beta N+2-\beta}(s-a)
\end{equation}
almost surely on $\Omega''$ for all $s\in(a,b)$, contradicting our assumption.

\noindent{}Altogether, $\left(T(s)\right)_{s\geq0}$ is almost surely continuous with the right initial conditions, and $\left(T(s)\right)_{s>0}$ is a strong solution of (\ref{dmpkeq}). 

\hspace{6mm}

To finish the proof of Theorem \ref{thmstartatone}, we have to see why $T$ is the unique strong solution to (\ref{dmpkeq}) with the given initial conditions. Assume $\left(\overline{T}(s)\right)_{s\geq0}$ was another continuous process solving (\ref{dmpkeq}) for $s>0$ and starting with $\overline{T}_k(0)=1$ for all $k$. For fixed $n$, as $T^{(n)}_k(0)<1$ for all $k$, and by continuity of $\overline{T}$ and $T^{(n)}$, there is a random, but almost surely positive $S$ such that still
\begin{equation}
 T^{(n)}_k(s)<\overline{T}_k(s)
\end{equation}
for $k=1,...,N$ on the event $\lbrace{S\geq s}\rbrace$. As $\overline{T}$ and $T^{(n)}$ are solutions of (\ref{dmpkeq}) for all positive $s$, we apply Lemma \ref{lmaconserveorder} to obtain $T^{(n)}_k(s)<\overline{T}_k(s)$ almost surely for all $s\geq0$, and taking the limit in $n$, also
\begin{equation}
 T_k(s)\leq\overline{T}_k(s)
\end{equation}
for all $k$ and $s$.
This means we can use the random variable 
\begin{equation}
 u=\sum_{k=1}^N\left|\overline{T}_k-T_k\right|=\sum_{k=1}^N\left(\overline{T}_k-T_k\right)
\end{equation}
to control the distance of the two solutions. We have

\beq
\dd u(t)=f\left(\overline{T}(t),T(t)\right)\dd t+\sigma\left(\overline{T}(t),T(t)\right)\dd B(t)
\eeq
for all $t>0$, with $B$ a Brownian motion and 
\begin{equation}
 \begin{split}
  \left|f\left(\overline{T}(t),T(t)\right)\right|&\leq C u(t)\\
 \left|\sigma\left(\overline{T}(t),T(t)\right)\right|&\leq{\tilde{C}}\sqrt{u(t)}
 \end{split}
\end{equation}
with $N$-dependent constants $C,\tilde{C}$. We pick an $s>0$ and estimate
\beq
\label{estimateforu}
\mean{u(t)}=\int_s^t\mean{f\left(\overline{T}(\tau),T(\tau)\right)}\dd \tau\leq\mean{u(s)}\exp\left(C(t-s)\right)
\eeq
for all $t\geq s$. Now as $u\leq N$, and $u(s)\rightarrow0$ as $s\rightarrow0$ due to continuity of $T$ and $\overline{T}$, we have by dominated convergence in (\ref{estimateforu})
\beq
\mean{\left|u(t)\right|}=0
\eeq
for all $t\geq0$, and again by continuity of $T$ and $\overline{T}$, $T$ and $\overline{T}$ are undistinguishable. Thus we have obtained the unique strong solution by our construction.

For the proof of weak uniqueness, we note that there is an equivalent of Lemma \ref{lmaconserveorder} that allows us to compare a weak and a strong solution with initial data ordered like (\ref{initialordering}), which allows us to repeat the above argument with $\left(\overline{T},B\right)$ a weak solution of the DMPK equation for positive times, while $\overline{T}_k(0)=1$ for all $k$.
\end{proof}

\subsection{Equality in distribution}
\label{equal_dist}

Now we are ready to prove that the DMPK equation actually describes the process of the transmission eigenvalues.

Let $\beta=1$ or $\beta=2$, and $(\Omega,\bbP,\caF)$ be a probability space with a filtration $\left(\caF_s\right)_{s\geq0}$, and, using Brownian motions with respect to this filtration define $N\times N$ matrix valued processes $\mfa_+$, $\mfa_-$, $\mfb$ and $\tmfa_+$, $\tmfa_-$, $\tmfb$ distributed according to (\ref{abfromdbm}), with $(\mfa_+,\mfa_-,\mfb)$ independent of $(\tmfa_+,\tmfa_-,\tmfb)$, and both triples obeying the dependence/independence structure described under equation (\ref{abfromdbm}).

As it is convenient for later calculations, note the component-wise distribution of the matrix blocks: $\mfa_+$, $\mfa_-$, $\tmfa_+$ and $\tmfa_-$ are all equal in distribution with $\mu\nu$ entries
\begin{equation}
 \label{mfa}
\mfa_{\mu\nu}(s)=\Bigg\lbrace\begin{tabular}{l l}\ensuremath{1/\sqrt{2N}(B^R_{\mu\nu}(s)+iB^I_{\mu\nu}(s))}& \ensuremath{\mu<\nu}\\
\ensuremath{i/\sqrt{N}(B^I_{\mu\mu}(s))}&\ensuremath{\mu=\mu}\\
\ensuremath{-\overline{\mfa_{\nu\mu}(s)}}&\ensuremath{\mu>\nu}
\end{tabular}
\end{equation}
for $\mu,\nu=1,...,N$, with all  components of$B^R$ and $B^I$ independent, standard real Brownian motions. The blocks $\mfb$ and $\tmfb$ are equal in distribution with entries distributed like
\begin{equation}
 \label{mfb_0}
\mfb_{\mu\nu}(s)=\Bigg\lbrace\begin{tabular}{l l l l}\ensuremath{1/\sqrt{2(N+1)}(\hat{B}^R_{\mu\nu}(s)+i\hat{B}^I_{\mu\nu}(s))}& \ensuremath{\mu<\nu}&&\\
\ensuremath{\sqrt{1/(N+1)}(\hat{B}^R_{\mu\mu}(s)+i\hat{B}^I_{\mu\mu}(s))}&\ensuremath{\mu=\mu}&\hspace{1cm}&\ensuremath{(\beta=1)}\\
\ensuremath{\overline{\mfb_{\nu\mu}(s)}}&\ensuremath{\mu>\nu}&&
\end{tabular}
\end{equation}
or
\begin{equation}
\label{mfb_mag}
 \mfb_{\mu\nu}(s)=1/\sqrt{2N}(\hat{B}^R_{\mu\nu}(s)+i\hat{B}^I_{\mu\nu}(s))\hspace{1cm}(\beta=2),
\end{equation}
again with all components of $\hat{B}^R$ and $\hat{B}^I$ independent, standard real Brownian motions.

Define $\left(\caL(s)\right)_{s\geq 0}$ and $\left(\caM(s)\right)_{s\geq 0}$ by
\begin{equation}
 \caL(s)=\left(\begin{array}{cc} \mathfrak{a}_+(s) & \mathfrak{b}(s) \\[1mm] \mathfrak{b}^*(s) &{ \mathfrak{a}_-}(s)\end{array}\right)\
\end{equation}
and
\begin{equation}
\label{evolM2}
 \dd\caM(s)=\dd\caL(s)\caM(s)
\end{equation}
with initial condition $\caM(0)=1_{2N}$. Let $\lambda^{\caM}(s)\in\bbR^N$ be eigenvalues of $\caM_{++}^*\caM_{++}(s)$, ordered by $\lambda\upm_1\geq...\geq\lambda\upm_N$ and define the transmission eigenvalues as before by $T\upm_k={\left(\lambda\upm_k\right)}^{-1}$. 

On the other hand, define $N$ independent Brownian motions by (\ref{defBk}), so
\begin{equation}
\label{defineBM}
 B_k(s)=-\sqrt{{\beta(N-1)+2}}\Re\tilde{\mfb}_{kk}(s),
\end{equation}
and let $\left(T_k(s)\right)_{s\geq0}$, $k=1,...,N$ be the solution of the DMPK equation (\ref{dmpkeq}) driven by those $B_k$, starting from $T_k(0)=1$ for all $k$, as constructed in Theorem \ref{thmstartatone}.

\begin{thm}
 \label{thm_dmpkdist}
The distributions of $\left(T\upm(s)\right)_{s\geq0}$ and $\left(T(s)\right)_{s\geq0}$ in $\left(C_{\bbR^N}[0,\infty),\caB\left(C_{\bbR^N}[0,\infty)\right)\right)$ are identical.
\end{thm}
\remark{}The Borel $\sigma$-field $\caB\left(C_{\bbR^N}[0,\infty)\right)$ is defined with respect to the metric of locally uniform convergence. For the statement of the theorem, we obviously do not need $T\upm$ and $T$ to be defined on the same probability space, but this will be very convenient for the proof.

The idea of the proof is similar to that for the Dyson Brownian Motion in \cite{agz}, essentially by reading the formal derivation of (\ref{dmpkeq}) backwards and making every step rigorous. To get started, note that we know that $T\upm(s)\in\overline{D_N}$, but even have
\begin{lma}
\label{lma_TMinDN}
 For $T\upm$ defined as above,
\begin{equation}
\label{TMinDN}
 \lim_{s\stackrel{>}{\rightarrow}0}\prob{T\upm(s)\in D_N}=1.
\end{equation} 
\end{lma}
\remark{}Once we have shown Theorem \ref{thm_dmpkdist}, we will actually know much more, namely
\begin{equation}
 \prob{\left\lbrace T\upm(s)\in D_N \forall s>0\right\rbrace}=1.
\end{equation}
For the moment, we morally need $\prob{T\upm(s)\in D_N}=1$ for any \emph{fixed} time $s>0$. This could be achieved by writing (\ref{evolM2}) as a Brownian motion on the manifold of transfer matrices, and then using the smoothness of the heat kernel \cite{dodziuk,grigoryan} to verify that the lower-dimensional manifold corresponding to $T\upm(s)\in\overline{D_N}\setminus D_N$ has probability zero. Instead, we use the weaker statement (\ref{TMinDN}) which can be proven directly and suffices as well.

\begin{proof}
 $T\upm(s)\in D_N$ is equivalent to $\lambda\upm\in (1,\infty)^N$ and nondegenerate, and finally, that the matrix $\caM_{-+}^*\caM_{-+}$ has positive, finite and nodegenerate eigenvalues. By the expansion 
 \begin{equation}
 \label{seriesM}
\caM(s)=1_{2N}+\int_0^s\dd\caL(t_1)+\sum_{l=2}^{\infty}\int_0^s\dd\caL(t_1)\int_0^{t_1}\dd\caL(t_2)...\int_0^{t_{l-1}}\dd\caL(t_l)
\end{equation}
we see that $\caM$ is certainly almost surely bounded and that for any $0\leq s<1$
\begin{equation}
 \caM_{-+}^*(s)\caM_{-+}(s)=\mfb(s)\mfb(s)^*+K(s)
\end{equation}
with $K(s)\in\bbC^{N\times N}$ hermitian and $\mean{\left\|K(s)\right\|}\leq C s^{3/2}$, while
\begin{equation}
 \mfb(s)\mfb(s)^*\stackrel{d}{=}s\mfb(1)\mfb(1)^*.
\end{equation}
Now once we see that $\mfb(1)\mfb(1)^*$ has almost surely positive and nondegenerate eigenvalues, we can denote the minimum of eigenvalue spacing and distance to zero of $\mathrm{spec}\left(\mfb(s)\mfb(s)^*\right)$ by $\delta(s)$, which is $\caO(s)$, and get
\begin{equation}
 \lim_{s\stackrel{>}{\rightarrow}0}\prob{\left\lbrace\left\|K(s)\right\|<\delta(s)/2\right\rbrace}=1
\end{equation}
by a Markov estimate. Then perturbation theory for the spectrum of hermitian matrices (\cite{kato}, Chapter V., Theorem 4.10) proves Lemma \ref{lma_TMinDN}.

To see that $\mfb(1)\mfb(1)^*$ is positive and nondegenerate almost surely, we proceed as in the proof of Lemma 2.5.5. in \cite{agz}. Let $p$ be the characteristic polynomial of $\mfb(1)\mfb(1)^*$, and $q(\lambda)=\lambda p(\lambda)$. $\mfb(1)\mfb(1)^*$ has a degenerate or zero eigenvalue if and only if $q$ has a multiple zero, which is when the discriminant $D(q)$ is zero (\cite{agz}, A.4). $D(q)$ is a polynomial in the coefficients of $q$, and thus a polynomial in the coefficients of $\mfb(1)$, which does not vanish identically (as there are of course nondegenerate positive realizations of $\mfb(1)\mfb(1)^*$). But as the probability density of the entries of $\mfb(1)$ is continuous, $\prob{D(q)=0}=0$, because zero sets of nonvanishing polynomials have Lebesgue measure zero.
\end{proof}

\noindent{}Thus we know that if we define for any $\eps>0$ the event
\begin{equation}
 A_\eps=\left\lbrace\omega\in\Omega:T\upm(\eps)\in D_N\right\rbrace
\end{equation}
then $\prob{A_\eps}\rightarrow 1$ as $\eps\rightarrow 0$. For the time being we fix $\eps>0$ and start the equation (\ref{dmpkeq}) on the set $A_\eps$ from $s=\eps$ with initial data $T^\eps(\eps)=T\upm(\eps)$ to obtain a unique strong solution $\left(T^\eps(s;\omega)\right)_{s\geq \eps, \omega\in A_\eps}$. The paths of $T^\eps$ stay inside $D_N$ for all $s\geq\eps$ almost surely on $A_\eps$ by Theorem \ref{thmnocoll}. To define $T^\eps$ for all times and on the whole probability space, let $T^\eps(s;\omega)=T\upm(s;\omega)$ if $0\leq s\leq\eps$ or $\omega\notin A_\eps$. $T^\eps$ is a process with almost surely continuous paths. Of course, $T^\eps$ gives rise to processes $\lambda^\eps_k=\left(T^\eps_k\right)^{-1}$. From the well-posedness of the DMPK equation for $T^\eps$ on $A_\eps$ for $s\geq\eps$, we see by an application of the It\^{o} formula that $\lambda^\eps$ solve (\ref{SDElambda}) on $A_\eps$ for $s\geq\eps$ with initial condition $\lambda^\eps(\eps)=\lambda\upm(\eps)$. As a consequence, the following $N\times N$ matrix-valued processes are well-defined and have continuous semimartingale entries: On $A_\eps$, start with $L^\pm_\eps(\eps)=0$, $R^\pm_\eps(\eps)=0$, and let for $s\geq\eps$
\begin{equation}
\label{defLLRR}
\begin{split}
 \dd L^+_{\mu\nu}&=\dd \tmfa_{+,\mu\nu}+\delta_{\mu\neq\nu}\left(\frac{\sqrt{\lambda_\nu^2-\lambda_\nu}\dd\tmfb_{\mu\nu}+\sqrt{\lambda_\mu^2-\lambda_\mu}\dd\overline{\tmfb_{\nu\mu}}}{\lambda_\nu-\lambda_\mu}\right)+i\delta_{\mu\nu}\frac{2\lambda_\nu-1}{2\sqrt{\lambda_\nu^2-\lambda_\nu}}\dd\Im\tmfb_{\nu\nu}\\
 \dd L^-_{\mu\nu}&=\dd \tmfa_{-,\mu\nu}+\delta_{\mu\neq\nu}\left(\frac{\sqrt{\lambda_\nu^2-\lambda_\nu}\dd\overline{\tmfb_{\nu\mu}}+\sqrt{\lambda_\mu^2-\lambda_\mu}\dd\tmfb_{\mu\nu}}{\lambda_\nu-\lambda_\mu}\right)-i\delta_{\mu\nu}\frac{2\lambda_\nu-1}{2\sqrt{\lambda_\nu^2-\lambda_\nu}}\dd\Im\tmfb_{\nu\nu}\\
 \dd R^+_{\mu\nu}&=\delta_{\mu\neq\nu}\left(\frac{\sqrt{\lambda_\mu(\lambda_\nu-1)}\dd\tmfb_{\mu\nu}+\sqrt{\lambda_\nu(\lambda_\mu-1)}\dd\overline{\tmfb_{\nu\mu}}}{\lambda_\mu-\lambda_\nu}\right)-i\delta_{\mu\nu}\frac{1}{2\sqrt{\lambda_\nu^2-\lambda_\nu}}\dd\Im\tmfb_{\nu\nu}\\
 \dd R^-_{\mu\nu}&=\delta_{\mu\neq\nu}\left(\frac{\sqrt{\lambda_\mu(\lambda_\nu-1)}\dd\overline{\tmfb_{\nu\mu}}+\sqrt{\lambda_\nu(\lambda_\mu-1)}\dd\tmfb_{\mu\nu}}{\lambda_\mu-\lambda_\nu}\right)+i\delta_{\mu\nu}\frac{1}{2\sqrt{\lambda_\nu^2-\lambda_\nu}}\dd\Im\tmfb_{\nu\nu},
\end{split}
\end{equation}
where we suppressed the $\eps$-dependence of $\lambda,L,R$ for a moment. Note that all four matrices are skew-hermitian, and furthermore, for $\beta=1$, also $L^+_\eps=\overline{L^-_\eps}$ and $R^+_\eps=\overline{R^-_\eps}$.

Now decompose $\caM(\eps)$ as in (\ref{singvalM}), and define matrices $U_\pm^\eps(s)$, $V_\pm^\eps(s)$ for $s\geq\eps$ on $A_\eps$ as the unique strong solutions of
\begin{equation}
 \begin{split}
  \dd U_\pm^\eps(s)&=U_\pm^\eps(s)\left(\dd L^\pm_\eps(s)-\frac{1}{2}\dd L^\pm_\eps(s)\dd L^\pm_\eps(s)^*\right),\hspace{1cm}U_\pm^\eps(\eps)=U_\pm(\eps)\\
\dd V_\pm^\eps(s)&=\left(\dd R^\pm_\eps(s)-\frac{1}{2}\dd R^\pm_\eps(s)^*\dd R^\pm_\eps(s)\right)V_\pm^\eps(s),\hspace{1cm}V_\pm^\eps(\eps)=V_\pm(\eps).
 \end{split}
\end{equation}
Existence and uniqueness of these solutions can be easily obtained as long as the denominators in the coefficients of (\ref{defLLRR}) do not get too singular, which in turn can be controlled by the stopping-time argument introduced in the proof of Theorem \ref{thmnocoll}. Furthermore, we have
\begin{lma}
 All four matrices $U_\pm^\eps$, $V_\pm^\eps$ are unitary. Also, for $\beta=1$, $\overline{U_+^\eps}={U_-^\eps}$ and $\overline{V_+^\eps}={V_-^\eps}$.
\end{lma}
\begin{proof}
 For the first statement, consider, say $U_\pm^\eps$. $U_\pm^\eps(\eps)=U_\pm(\eps)$ is unitary by definition. On $A_\eps$, It\^{o}'s formula yields for all $s\geq\eps$
\begin{equation}
\begin{split}
 \dd\left(U_\pm^{\eps}U_\pm^{\eps*}\right)&=U_\pm^\eps\left(\dd L^\pm_\eps-\frac{1}{2}\dd L^\pm_\eps\dd L^{\pm*}_\eps\right)U_\pm^{\eps*}+U_\pm^\eps\left(\dd L^{\pm*}_\eps-\frac{1}{2}\dd L^\pm_\eps \dd L^{\pm*}_\eps\right)U_\pm^{\eps*}\\&\hspace{2cm}+U_\pm^\eps\dd L^\pm_\eps\dd L^{\pm*}_\eps U_\pm^{\eps*}\\
&=0
\end{split}
\end{equation}
by the skew-symmetry of $L^\pm_\eps$. In addition, if $\beta=1$, we have 
\begin{equation}
\overline{U_+^\eps(\eps)}=\overline{U_+(\eps)}={U_-}(\eps)={U_-^\eps}(\eps),
\end{equation}
and, by $\overline{L^+_\eps}={L^-_\eps}$,
\begin{equation}
 \dd\overline{ U_+^\eps}=\overline{U_+^\eps}\left(\dd \overline{L^+_\eps}-\frac{1}{2}\dd \overline{L^+_\eps}\dd \overline{L^+_\eps}^*\right)=\overline{U_+^\eps}\left(\dd {L^-_\eps}-\frac{1}{2}\dd {L^-_\eps}\dd {L^{-*}_\eps}\right),
\end{equation}
so $\overline{U_+^\eps}$ solves the SDE for ${U_-^\eps}$. The proof for the $V_\pm^\eps$ matrices is analogous.
\end{proof}

On $A_\eps$ for $s\geq\eps$ we now write
\begin{equation}
\begin{split}
 \caU^\eps&(s)=\begin{pmatrix} U^\eps_{+}(s)&0\\0&U^\eps_{-}(s)\end{pmatrix}\\
 \caV^\eps&(s)=\begin{pmatrix} V^\eps_{+}(s)&0\\0&V^\eps_{-}(s)\end{pmatrix}\\
 \caS^\eps&(s)=\begin{pmatrix} \Lambda^\eps(s)&\sqrt{\Lambda^\eps(s)^2-1}\\\sqrt{\Lambda^\eps(s)^2-1}&\Lambda^\eps(s)\end{pmatrix}
\end{split}
\end{equation}
with $\Lambda^\eps$ a diagonal matrix with entries $\Lambda^\eps_{\nu\nu}=\sqrt{\lambda^\eps_\nu}$. Now we are ready to replace $\caM$ by the process $\caM^\eps$ with
\begin{equation}
 \begin{split}
  \caM^\eps(s;\omega)&=\caM(s;\omega)\hspace{2.5cm}\mbox{for }0\leq s\leq \eps\mbox{ or }\omega\notin A_\eps\\
\caM^\eps(s;\omega)&=\caU^\eps(s;\omega)\caS^\eps(s;\omega)\caV^\eps(s;\omega)\hspace{1cm}\mbox{otherwise. }
 \end{split}
\end{equation}
Also, let
\begin{equation}
 \tilde{\caL}(s)=\left(\begin{array}{cc} \tilde{\mathfrak{a}}_+(s) & \tilde{\mathfrak{b}}(s) \\[1mm] \tilde{\mathfrak{b}}^*(s) &{ \tilde{\mathfrak{a}}_-}(s)\end{array}\right).
\end{equation}
Now as $\caU^\eps(s)$, $\caS^\eps(s)$, $\caV^\eps(s)$ are all well defined on $A_\eps$ as semi-martingales for $s\geq\eps$, we can calculate the dynamics of $\caM^\eps(s)$ with the It\^{o} formula. In the calculation below, all equations hold as stochastic differential equations on the event $A_\eps$ starting from time $\eps$. Suppressing the $\eps$- and $s$- dependence, denote
\begin{equation}
 \caA=\begin{pmatrix} L^+&0\\0&L^-\end{pmatrix}\hspace{12mm}\caB=\begin{pmatrix} R^+&0\\0&R^-\end{pmatrix},
\end{equation}
so that
\begin{equation}
 \dd\caU=\caU\left(\dd\caA+\frac{1}{2}\dd\caA\dd\caA\right)\hspace{8mm}\dd\caV=\left(\dd\caB+\frac{1}{2}\dd\caB\dd\caB\right)\caV,
\end{equation}
where we used the skew-symmetry of $\caA$ and $\caB$. As $\dd\caS$ is a matrix consisting of four diagonal blocks, we only give the $\nu$-th diagonal element of each block, reading
\begin{equation}
 \left(\dd\caS\right)_\nu=\begin{pmatrix} \frac{\dd\lambda_\nu}{2\sqrt{\lambda_\nu}}&\frac{\dd\lambda_\nu}{2\sqrt{\lambda_\nu-1}}\\\frac{\dd\lambda_\nu}{2\sqrt{\lambda_\nu-1}}&\frac{\dd\lambda_\nu}{2\sqrt{\lambda_\nu}}\end{pmatrix}-\frac{1}{2\beta(N-1)+4}\begin{pmatrix}\frac{\lambda_\nu-1}{\sqrt{\lambda_\nu}}&\frac{\lambda_\nu}{\sqrt{\lambda_\nu-1}}\\\frac{\lambda_\nu}{\sqrt{\lambda_\nu-1}}&\frac{\lambda_\nu-1}{\sqrt{\lambda_\nu}}\end{pmatrix}\dd s.
\end{equation}
After splitting $\caS$ into its continuous local martingale part $\caS_1$ and a finite variation part $\caS_2$, we have
\begin{equation}
\label{splitS}
 \begin{split}
  \left((\dd\caS_1)\caS^{-1}\right)_\nu&=\begin{pmatrix} 0&\dd\Re\tmfb_{\nu\nu}\\\dd\Re\tmfb_{\nu\nu}&0\end{pmatrix}
\end{split}
\end{equation}
and
\begin{equation}
\label{dS2}
\begin{split}
&\left((\dd\caS_2)\caS^{-1}\right)_\nu\\&\hspace{0.5cm}=\begin{pmatrix} 0&\frac{1}{2{\sqrt{\lambda_\nu^2-\lambda_\nu}}}\\\frac{1}{2{\sqrt{\lambda_\nu^2-\lambda_\nu}}}&0\end{pmatrix}\left(2\lambda_\nu-1+\frac{\beta}{\beta(N-1)+2}\sum_{\rho\neq \nu}\frac{2\lambda_\rho\lambda_\nu-\lambda_\nu-\lambda_\rho}{\lambda_\nu-\lambda_\rho}\right)\dd s\\&\hspace{1cm}+\frac{1}{2\beta(N-1)+4}\begin{pmatrix}1&-\frac{2\lambda_\nu-1}{\sqrt{\lambda_\nu^2-\lambda_\nu}}\\-\frac{2\lambda_\nu-1}{\sqrt{\lambda_\nu^2-\lambda_\nu}}&1\end{pmatrix}\dd s\\
 &\hspace{5mm}=\frac{1}{2\beta(N-1)+4}\begin{pmatrix} 1&0\\0&1\end{pmatrix}\dd s\\
&\hspace{1cm}+\frac{1}{\beta(N-1)+2}\begin{pmatrix} 0&\frac{2\lambda_\nu-1}{2\sqrt{\lambda_\nu^2-\lambda_\nu}}+\beta\sum_{\rho\neq\nu}\frac{\sqrt{\lambda_\nu^2-\lambda_\nu}}{\lambda_\nu-\lambda_\rho}\\\frac{2\lambda_\nu-1}{2\sqrt{\lambda_\nu^2-\lambda_\nu}}+\beta\sum_{\rho\neq\nu}\frac{\sqrt{\lambda_\nu^2-\lambda_\nu}}{\lambda_\nu-\lambda_\rho}&0\end{pmatrix}\dd s
\end{split}
\end{equation}
where a $\caS^{-1}$ factor was included for later convenience. Comparing (\ref{splitS}) to (\ref{defLLRR}), it is clear that the bracket processes $\dd\caA\dd\caS=\dd\caB\dd\caS=0$ vanish, and thus the It\^{o} formula for $\caM^\eps$ reads
\begin{equation}
\label{calcdM}
\begin{split}
 \dd\caM^\eps&=\dd(\caU\caS\caV)=(\dd\caU)\caS\caV+\caU(\dd\caS)\caV+\caU\caS(\dd\caV)+(\dd\caU)\caS(\dd\caV)\\
&=\caU(\dd\caA)\caS\caV+\caU(\dd\caS_1)\caV+\caU\caS(\dd\caB)\caV\\
&\hspace{1cm}+\caU(\dd\caS_2)\caV+\frac{1}{2}\caU(\dd\caA\dd\caA)\caS\caV+\frac{1}{2}\caU\caS(\dd\caB\dd\caB)\caV+
\caU(\dd\caA)\caS(\dd\caB)\caV.
\end{split}
\end{equation}
By a straightforward computation using (\ref{defLLRR}) and (\ref{splitS}), we see for the second line
\begin{equation}
 \label{martpart}
\caU(\dd\caA)\caS\caV+\caU(\dd\caS_1)\caV+\caU\caS(\dd\caB)\caV=\caU(\dd\tilde{\caL})\caS\caV.
\end{equation}
For the third line of (\ref{calcdM}), we have the following trick
\begin{equation}
 (\dd\caA)\caS\dd\caB=(\dd\caA)(\dd\caS_1+\caS\dd\caB)=(\dd\caA)(\dd\tilde{\caL}-\dd\caA)\caS
\end{equation}
by (\ref{martpart}), and analogously,
\begin{equation}
 \caS(\dd\caB\dd\caB)=(\dd\caS_1+\caS\dd\caB)(\dd\caB)=(\dd\tilde{\caL}-\dd\caA)\caS\dd\caB=(\dd\tilde{\caL}-\dd\caA)(\dd\tilde{\caL}-\dd\caA)\caS-\dd\tilde{\caL}\dd\caS_1.
\end{equation}
So far,
\begin{equation}
\label{sofar}
\begin{split}
 \dd\caM^\eps=\caU&(\dd\tilde{\caL})\caS\caV\\&+\frac{1}{2}\caU(\dd\tilde{\caL}\dd\tilde{\caL})\caS\caV+\caU\left((\dd\caS_2)\caS^{-1}+\frac{1}{2}(\dd\caA\dd\tilde{\caL}-\dd\tilde{\caL}\dd\caA)-\frac{1}{2}\dd\tilde{\caL}(\dd\caS_1)\caS^{-1}\right)\caS\caV.
\end{split}
\end{equation}
By definition of $\tmfa_+,\tmfa_-,\tmfb$, $\dd\tilde{\caL}\dd\tilde{\caL}=0$, while
\begin{equation}
 -\frac{1}{2}\dd\tilde{\caL}(\dd\caS_1)\caS^{-1}=-\frac{1}{2\beta(N-1)+4}\begin{pmatrix} 1_N&0\\0&1_N\end{pmatrix}\dd s
\end{equation}
and
\begin{equation}
\begin{split}
 \frac{1}{2}(\dd\caA\dd\tilde{\caL}&-\dd\tilde{\caL}\dd\caA)=\frac{1}{2}\left(\begin{pmatrix} \dd L^+&0\\0&\dd L^-\end{pmatrix}\begin{pmatrix}0&\dd\tmfb\\\dd\tmfb^*&0\end{pmatrix}+\begin{pmatrix}0&\dd\tmfb\\\dd\tmfb^*&0\end{pmatrix}\begin{pmatrix} \dd L^{+*}&0\\0&\dd L^{-*}\end{pmatrix}\right)\\
&\hspace{-1cm}=-\frac{1}{\beta(N-1)+2}\begin{pmatrix} 0&\frac{2\lambda_\nu-1}{2\sqrt{\lambda_\nu^2-\lambda_\nu}}+\beta\sum_{\rho\neq\nu}\frac{\sqrt{\lambda_\nu^2-\lambda_\nu}}{\lambda_\nu-\lambda_\rho}\\\frac{2\lambda_\nu-1}{2\sqrt{\lambda_\nu^2-\lambda_\nu}}+\beta\sum_{\rho\neq\nu}\frac{\sqrt{\lambda_\nu^2-\lambda_\nu}}{\lambda_\nu-\lambda_\rho}&0\end{pmatrix}\dd s.
\end{split}
\end{equation}
The matrix in the last line consists of four diagonal blocks again, and we have only given the $\nu$-th diagonal elements.

So all terms in (\ref{sofar}) cancel except for one, and we have on $A_\eps$, for $s\geq \eps$
\begin{equation}
\label{DMPKonA}
 \dd \caM^\eps(s)=\caU^\eps(s)\dd\tilde{\caL}(s)\caS^\eps(s)\caV^\eps(s)=\caU^\eps(s)\dd\tilde{\caL}(s)\caU^{\eps*}(s)\caM^\eps(s).
\end{equation}
Define the process $\caL^\eps$ by
\begin{equation}
 \begin{split}
  \caL^\eps(s;\omega)&=\caL(s;\omega)\hspace{3.5cm}\mbox{for }0\leq s\leq \eps\mbox{ or }\omega\notin A_\eps\\
\caL^\eps(s)&=\caL(\eps)+\int_\eps^s\caU^\eps(t)\dd\tilde{\caL}(t)\caU^{\eps*}(t)\hspace{1cm}\mbox{otherwise. }
 \end{split}
\end{equation}
The stochastic integral in the second line is well-defined by the construction we have followed so far, so $\caL^\eps$ is a well-defined continuous local martingale, and in view of (\ref{DMPKonA}) we have on $\Omega$ and for all times $s\geq0$
\begin{equation}
 \dd\caM^\eps(s)=\dd\caL^\eps(s)\caM^\eps(s).
\end{equation}
To verify that the process $\caM^\eps$ has the same law in the space of continuous functions as $\caM$, it is enough to show
\begin{equation}
 \left(\caL^\eps(s)\right)_{s\geq0}\stackrel{d}{=}\left(\caL(s)\right)_{s\geq0}.
\end{equation}
For $s\leq \eps$, we even have $\caL^\eps(s)=\caL(s)$, and for $s\geq\eps$, we need to prove that the paths of
\begin{equation}
 \caL^\eps(s)-\caL^\eps(\eps)=1(A_\eps)\int_\eps^s\caU^\eps(t)\dd\tilde{\caL}(t)\caU^{\eps*}(t)+1(A_\eps^c)(\caL(s)-\caL(\eps))
\end{equation}
are distributed as $\caL(s)-\caL(\eps)$, and independent of $\left(\caL^\eps(\tau)\right)_{\tau\leq\eps}$. For the first statement, it is enough to notice that the bracket processes of $\caL^\eps$ and $\caL$ coincide by the invariance property
\begin{equation}
 \label{incrLinvar}
\caU^\eps(t)\left(\tilde{\caL}(t+\delta)-\tilde{\caL}(t)\right)\caU^{\eps*}(t)\stackrel{d}{=}\tilde{\caL}(t+\delta)-\tilde{\caL}(t)
\end{equation}
for all $\delta>0$, which follows from the distribution of $\tmfa_+,\tmfa_-,\tmfb$.
For the independence, it is important to note that $\caU^\eps(t)$ do depend on $\left(\caL^\eps(\tau)\right)_{\tau\leq\eps}$ by their initial condition, but this dependence factors out, again by (\ref{incrLinvar}). So finally we have
\begin{equation}
 \left(\caM^\eps(s)\right)_{s\geq0}\stackrel{d}{=}\left(\caM(s)\right)_{s\geq0}.
\end{equation}
The process $T^\eps$ is by construction the transmission eigenvalue process of $\caM^\eps$, and thus
\begin{equation}
\label{Teps=TM}
 \left({T}^{\eps}(s)\right)_{s\geq0}\stackrel{d}{=}\left(T\upm(s)\right)_{s\geq0},
\end{equation}
on path space. $T\upm$ and $T^\eps$ are almost surely continuous processes, as $\caM$ and $\caM^\eps$ are and singular values of a (sub)matrix depend continuously on the matrix. As the ``largest part'' of $T^\eps$ is already defined as a solution of the DMPK equation, we are now in a position to compare the distributions of $T\upm$ and $T$, and thus to prove Theorem \ref{thm_dmpkdist}.
\begin{proof}
We will prove that the distribution of $\left({T}^\eps(s)\right)_{s\geq0}$ on $\left(C_{\bbR^N}[0,\infty),\caB\left(C_{\bbR^N}[0,\infty)\right)\right)$ converges to that of $\left({T}(s)\right)_{s\geq0}$, as $\eps\rightarrow0$. As we know that all $T^\eps$ are equal to $T\upm$ in distribution on path space, this can only mean that the distributions of $\left({T}\upm(s)\right)_{s\geq0}$ and $\left({T}(s)\right)_{s\geq0}$ already coincide, which proves Theorem \ref{thm_dmpkdist}.

\noindent{}From the initial condition for $T$, its boundedness and its continuity, we directly have 
\begin{equation}
 \mean{\sup_{0\leq s\leq \eps}\sum_{k=1}^N\left|1-T_k(s)\right|}\rightarrow0\hspace{1cm}(\eps\rightarrow0),
\end{equation}
and, recalling that $T^\eps(s)=T\upm(s)$ for $s\leq\eps$, and thus $\eps$-independent, bounded, and continuous, also
\begin{equation}
\label{Tepssmalltimes}
 \mean{\sup_{0\leq s\leq \eps}\sum_{k=1}^N\left|1-T^\eps_k(s)\right|}\rightarrow0\hspace{1cm}(\eps\rightarrow0).
\end{equation}
Thus we already know
\begin{equation}
\label{dist1}
 \mean{\sup_{0\leq s\leq \eps}\sum_{k=1}^N\left|T_k(s)-T^\eps_k(s)\right|}\rightarrow0\hspace{1cm}(\eps\rightarrow0).
\end{equation}
 For a given $\eps>0$, define the auxiliary process $\left(\tilde{T}^\eps(s)\right)_{s\geq\eps}$ as the unique strong solution of a DMPK equation driven by $B_k(s)$, $k=1,...,N$, $s\geq\eps$, starting from $\tilde{T}^\eps_k(\eps)=1$  for all $k$, which exists by Theorem \ref{thmstartatone}. On the other hand, for any $\eps>0$, $T^\eps(\eps)\in D_N$ on the event $A_\eps$, so we have the ordering $\tilde{T}^\eps_k(s)>T^\eps(s)$ for all $s\geq\eps$ on $A_\eps$ by Lemma \ref{lmaconserveorder} (here and in the following we ignore the singular initial condition for $\tilde{T}^\eps$, as we have seen in the proof of Theorem \ref{thmstartatone} how to approximate $\tilde{T}^\eps$ from below). Thus on $A_\eps$, we have 
\begin{equation}
 u_\eps(s)=\sum_{k=1}^N\left|\tilde{T}^\eps_k(s)-T^\eps_k(s)\right|=\sum_{k=1}^N\left(\tilde{T}^\eps_k(s)-T^\eps_k(s)\right)
\end{equation}
for all $s\geq0$, and can argue as at the end of the proof of Theorem \ref{thmstartatone}, that on event $A_\eps$, $u_\eps$ is a solution of
\begin{equation}
 u_\eps(s)=u_\eps(\eps)+\int_\eps^sf\left(\tilde{T}^\eps_k(t),T^\eps_k(t)\right)\dd t+\int_\eps^s\sigma\left(\tilde{T}^\eps_k(t),T^\eps_k(t)\right)\dd B^\eps(t)
\end{equation}
with a Brownian motion $B^\eps$ and 
\begin{equation}
\begin{split}
 \left|f\left(\tilde{T}^\eps_k(t),T^\eps_k(t)\right)\right|&\leq C u_\eps(t),\\
\left|\sigma\left(\tilde{T}^\eps_k(t),T^\eps_k(t)\right)\right|&\leq \tilde{C} \sqrt{u_\eps(t)},
\end{split}
\end{equation}
with $C, \tilde{C}$ only depending on $N$.
Thus we have again
\begin{equation}
 \mean{u_\eps(s)\cdot1(A_\eps)}\leq \mean{u_\eps(\eps)\cdot1(A_\eps)}\exp(C(s-\eps)).
\end{equation}
Fixing an $t>\eps$, with an application of the Doob inequality
\begin{equation}
\begin{split}
 &\mean{\sup_{s\in[\eps,t]}\left(1(A_\eps)\cdot\int_\eps^s\sigma\left(\tilde{T}^\eps_k(\tau),T^\eps_k(\tau)\right)\dd B^\eps(\tau)\right)^2}\\
&\hspace{1cm}\leq4\mean{\left(1(A_\eps)\cdot\int_\eps^t\sigma\left(\tilde{T}^\eps_k(\tau),T^\eps_k(\tau)\right)\dd B^\eps(\tau)\right)^2}\\
&\hspace{1cm}\leq4\tilde{C}^2\mean{1(A_\eps)\cdot\int_\eps^tu_\eps(\tau)\dd \tau}\\
&\hspace{1cm}\leq4\tilde{C}^2\mean{1(A_\eps)u_\eps(\eps)}\int_\eps^t\exp(C(\tau-\eps))\dd \tau\\
&\hspace{1cm}\leq\frac{4\tilde{C}^2}{C}e^{Ct}\mean{1(A_\eps)u_\eps(\eps)}
\end{split}
\end{equation}
and the estimate for $f$, we even find a constant $K(t,N)$ such that
\begin{equation}
 \mean{\sup_{s\in[\eps,t]}u_\eps(s)\cdot 1(A_\eps)}\leq K(t,N)\left(\mean{1(A_\eps)\cdot u_\eps(\eps)}+\sqrt{\mean{1(A_\eps)\cdot u_\eps(\eps)}}\right).
\end{equation}
Together with Lemma \ref{lma_TMinDN}, and our control of $u_\eps(\eps)$ from (\ref{Tepssmalltimes}),
\begin{equation}
\label{dist2}
 \mean{\sup_{s\in[\eps,t]}\sum_{k=1}^N\left|\tilde{T}^\eps_k(s)-T^\eps_k(s)\right|}\leq K(t,N)\left(\mean{ u_\eps(\eps)}+\sqrt{\mean{ u_\eps(\eps)}}\right)+N(1-\prob{A_\eps})\rightarrow0
\end{equation}
as $\eps\rightarrow0$ for any fixed $t$. By a similar argument, (except that we do not have to control a possible ill-behaved event $A_\eps^c$),
\begin{equation}
\label{dist3}
\mean{\sup_{s\in[\eps,t]}\sum_{k=1}^N\left|\tilde{T}^\eps_k(s)-T_k(s)\right|}\rightarrow0
\end{equation}
as $\eps\rightarrow0$, $t$ fixed. But the combination of (\ref{dist1}), (\ref{dist2}), (\ref{dist3}) is equivalent to 
\begin{equation}
 \mean{\mathrm{dist}\left(T^\eps,T\right)}\rightarrow0
\end{equation}
as $\eps\rightarrow0$, with $\mathrm{dist}$ denoting the metric of local uniform convergence on $C_{\bbR^N}[0,\infty)$. A fortiori, $T^\eps$ converges to $T$ in distribution on $C_{\bbR^N}[0,\infty)$.
\end{proof}

\subsection{General stochastic processes with Coulomb repulsion}
In general, we expect the following to be true for stochastic processes with Coulomb repulsion.

Let $N\in\mathbb{N}$, and choose bounded $c_{kl}:\mathbb{R}^N\rightarrow \mathbb{R}$ twice differentiable with bounded derivatives, $v:\mathbb{R}^N\rightarrow\mathbb{R}^N$ Lipschitz, $D_k:\mathbb{R}^N\rightarrow \mathbb{R}$ Lipschitz for all $k,l=1,...,N$. Moreover, let the following symmetry conditions be satisfied:
\begin{equation}
\label{ext_sym}
\begin{split}
c_{\sigma(k)\sigma(l)}(x_1,...,x_N)&=c_{kl}\left(x_{\sigma(1)},...x_{\sigma(N)}\right)\\
v_{\sigma(k)}(x_1,...,x_N)&=v\left(x_{\sigma(1)},...x_{\sigma(N)}\right)
\end{split}
\end{equation}
for all permutations $\sigma$ and all $x\in\mathbb{R}^N$, and assume
\begin{equation}
\label{ext_repulsion}
c_{kl}(x)+c_{lk}(x)\geq\frac{1}{2}\left(D_k^2(x)+D_l^2(x)\right)
\end{equation}
for all $k,l=1,...,N$, and all $x\in\mathbb{R}^N$ with $|x_k-x_l|$ sufficiently small.

Then, for any initial data $x_1(0)<...< x_N(0)$, the stochastic differential equation 
\begin{equation}
\label{ext_sde}
{\dd}x_k=D_k(x)\dd B_k+v_k(x)\dd t+\sum_{l\neq k}\frac{c_{kl}(x)}{x_k-x_l}\dd t
\end{equation}
has a unique strong solution $\left(X_t\right)_{t\geq0}$, with $x_1(t)<...< x_N(t)$ for all $t\geq0$ almost surely.

If, even more, $c_{kl}(x)>0$ for all $k,l,x$, one can also start from singular intial conditions $x_1(0)\leq...\leq x_N(0)$ and still have a unique continuous stochastic process $\left(X_t\right)_{t\geq0}$ such that for all positive $t$ the particles are strictly ordered $x_1(t)<...< x_N(t)$, and $\left(X_t\right)_{t>0}$ is a strong solution to (\ref{ext_sde}).

A short comment on the assumptions is in order. While (\ref{ext_repulsion}) makes sure that the diffusion is controlled by level repulsion at all times, and can essentially not be relaxed without allowing for collisions, (\ref{ext_sym}) is just one of many ways to make sure that the "ODE part" of (\ref{ext_sde}) does not allow for intersecting trajectories. We choose to write down these assumptions as symmetries, as those should arise quite naturally whenever the $x_k$ are eigenvalues of a matrix-valued process with invariance properties similar to (\ref{eq_dLinvar}).
\section*{Acknowledgements}
It is a pleasure to thank Herbert Spohn for important advice at all stages of writing this paper. I also benefited a lot from discussions with members of Antti Kupiainen's group at Helsinki University, and I gratefully acknowledge financial support by the Academy of Finland during my stay there.

\end{document}